\documentclass[11pt,  reqno]{amsart}
\usepackage{amsmath,amssymb,amscd,amsthm}

\usepackage{color}
\usepackage{hyperref}

\usepackage{cases}

\usepackage{bm}
\usepackage{comment}

\allowdisplaybreaks[3]

\newtheorem{theorem}{Theorem} [section]

\newtheorem{lemma}[theorem]{Lemma}
\newtheorem{proposition}[theorem]{Proposition}
\newtheorem{remark}[theorem]{Remark}
\newtheorem{definition}[theorem]{Definition}
\newtheorem{corollary}[theorem]{Corollary}

\DeclareMathOperator*{\supp}{supp}

\DeclareMathOperator*{\esssup}{ess \ sup}

\newcommand{\noi}{\noindent}
\newcommand{\Z}{\mathbb{Z}}
\newcommand{\R}{\mathbb{R}}
\newcommand{\C}{\mathbb{C}}
\newcommand{\T}{\mathbb{T}}

\newcommand{\al}{\alpha}
\newcommand{\be}{\beta}
\newcommand{\dl}{\delta}

\newcommand{\eps}{\varepsilon}

\newcommand{\G}{\Gamma}

\newcommand{\Ld}{\Lambda}
\newcommand{\s}{\sigma}

\newcommand{\ft}{\widehat}
\newcommand{\Ft}{{\mathcal{F}}}
\newcommand{\wt}{\widetilde}
\newcommand{\cj}{\overline}
\newcommand{\dx}{\partial_x}

\newcommand{\dt}{\partial_t}

\newcommand{\pd}{\partial}

\newcommand{\ta}{\theta}

\renewcommand{\l}{\ell}

\newcommand{\les}{\lesssim}
\newcommand{\ges}{\gtrsim}

\newcommand{\jb}[1]
{\langle #1 \rangle}

\newcommand{\Nl}{\mathcal{N}}

\newcommand{\N}{\mathbb{N}}

\newcommand{\Al}{\mathcal A}
\newcommand{\Gl}{\mathcal G}

\newcommand{\fR}{\mathfrak R}

\newcommand{\kk}{\kappa}

\newcommand{\I}{\mathrm{I}}
\newcommand{\II}{\mathrm{I \! I}}

\newcommand{\mm}{{\bf m}}

\newcommand{\wsl}{\overset{*}{\rightharpoonup}}

\numberwithin{equation}{section}
\numberwithin{theorem}{section}

\let\Re=\undefined\DeclareMathOperator*{\Re}{Re}
\let\Im=\undefined\DeclareMathOperator*{\Im}{Im}

\date{\today}

\begin{document}
\baselineskip = 14pt

\title[A sharp criterion for semilinear Schr\"odinger equations]
{A sharp criterion for well-posedness
\\
of semilinear Schr\"odinger equations
\\
on the real line}

\author[T.~Kondo and M.~Okamoto]
{Toshiki Kondo and Mamoru Okamoto}

\address{
Toshiki Kondo\\
Department of Mathematics\\\
Graduate School of Advanced Science and Engineering\\
Hiroshima University\\
\hfill\break 1-3-1 Kagamiyama, Higashi-Hiroshima, 739-8526 Japan}
\email{d263049@hiroshima-u.ac.jp}

\address{
Mamoru Okamoto\\
Department of Mathematics\\
Graduate School of Advanced Science and Engineering\\
Hiroshima University\\
1-3-1 Kagamiyama,
Higashi-Hiroshima, 739-8526
Japan}
\email{mokamoto@hiroshima-u.ac.jp}



\begin{abstract}
We consider the Cauchy problem for semilinear Schr\"odinger equations on the real line.
We establish a sharp necessary and sufficient condition on the nonlinearity for local well-posedness in Sobolev spaces of regularity greater than five halves.
This criterion is formulated in terms of the Mizohata condition and is equivalent to explicit algebraic conditions on the nonlinearity.
To prove sufficiency,
we combine energy estimates with a gauge transformation adapted to the bounded primitive arising from the Mizohata condition.
To prove necessity,
we show that the failure of the condition yields a dense set of initial data for which no solution exists on any positive time interval.
\end{abstract}

\maketitle

\tableofcontents

\newpage

\section{Introduction}
\label{sec:intro}

We consider the following Cauchy problem for semilinear Schr\"odinger equations:
\begin{equation}
\left\{
\begin{aligned}
&\dt u + i \dx^2 u = F(u,\dx u,\cj u,\cj{\dx u}),\\
&u|_{t=0} = \phi,
\end{aligned}
\right.
\label{NLS}
\end{equation}
where
$u = u(t,x) :\R \times \R \to \C$ is an unknown function,
$\phi : \R \to \C$ is a given function,
and
$F\in C^\infty(\C^2;\C)$.%
\footnote{Throughout this paper,
$F \in C^\infty(\C^2;\C)$
means that
$F (\al,\be, \cj \al, \cj \be)$ is smooth with respect to $\Re \al$, $\Im \al$, $\Re \be$, and $\Im \be$
for $\al, \be \in \C$.
}
We write $F$ in the form
$F(\al,\be,\cj\al,\cj\be)$,
where
$\cj \al$ denotes the complex conjugate of $\al \in \C$.

For $\al = a + ib \in \C$ ($a,b \in \R$), we write
\[
\frac {\pd F}{\pd \al} := \frac 12 \Big(\frac{\pd F}{\pd a} - i \frac {\pd F}{\pd b} \Big), \quad
\frac {\pd F}{\pd \cj \al} := \frac 12 \Big(\frac{\pd F}{\pd a} + i \frac {\pd F}{\pd b} \Big).
\]
For simplicity, we use the following notation:
\[
F_\al = \frac {\pd F}{\pd \al}, \quad F_\be = \frac{\pd F}{\pd \be},
\quad F_{\cj \al} = \frac {\pd F}{\pd \cj \al}, \quad \cdots
\]
We abbreviate $F(0,0,0,0)$ to $F(0)$.

In the linear setting, let $a$ and $b$ be smooth complex-valued functions on $\R$ with all derivatives bounded.
The Cauchy problem
\[
\left\{
\begin{aligned}
&\dt u + i \dx^2 u = a(x)u + b(x)\dx u, \\
&u|_{t=0}=\phi
\end{aligned}
\right.
\]
is well-posed in $L^2(\R)$ if and only if
\[
\sup_{x\in\R}
\bigg| \int_0^x \Im b(y) dy \bigg| < \infty.
\]
This condition is known as the Mizohata condition.
See \cite{Mi85}; see also \cite{Tak80, Ich84}.

When $F$ is a polynomial consisting only of terms of degree at least three,
Kenig, Ponce, and Vega \cite{KPV93} proved small-data local well-posedness of \eqref{NLS} in $H^s(\R)$ for $s \ge \frac 72$.
Hayashi and Ozawa \cite{HaOz94} removed the smallness assumption on the initial data and improved the regularity threshold from $s \ge \frac 72$ to $s \ge 3$.
In addition, Pornnopparath \cite{Por18} showed small-data local well-posedness in $H^s(\R)$ for $s\ge \frac 32$.
For
$F\in C^\infty(\C^2;\C)$ satisfying
\[
F(\al, \be, \cj \al, \cj \be) = O (|\al|^3 + |\be|^3)
\]
near the origin,
Chihara \cite{Chi98} showed local existence and uniqueness in $H^m(\R)$ for integers $m\ge4$.
For further results on derivative nonlinear Schr\"odinger equations,
see \cite{Taka99, CKSTT02, AmSi15, HaOz16, HGKNV26} and the references therein.

On the other hand,
when the quadratic part of the Taylor expansion of $F$ at the origin does not vanish,
the situation becomes more delicate.
Indeed, Christ \cite{Chr} proved that the Cauchy problem for
\[
\dt u + i \dx^2 u = u \dx u
\]
is ill-posed in $H^s(\R)$ for every $s \in \R$.
See also \cite{KoOk26}.
In contrast, Gr\"unrock \cite{Gr00} showed that the Cauchy
problem for
\[
\dt u + i \dx^2 u = \cj{u} \dx \cj{u}
\]
is globally well-posed in $H^s(\R)$ for $s \ge 0$.
Moreover, Akase \cite{Ak25} showed that, for $c \in \C$, the Cauchy problem for
\[
\dt u + i \dx^2 u = i \dx (c u^2 - 2 \cj c |u|^2)
\]
is globally well-posed in $H^s(\R)$ for $s \ge 1$.
These results demonstrate that well-posedness in Sobolev spaces depends crucially on the structure of the quadratic derivative nonlinearity.

For more general quadratic derivative nonlinearities, 
local well-posedness has also been studied in weighted Sobolev spaces \cite{KPV93,Chi94,HaOz94,Chi98}.
Moreover,
for a certain class of quadratic derivative nonlinearities,
Ozawa \cite{Oz98}
proved local well-posedness
for initial data $f\in H^1(\R)$ satisfying
\[
\sup_{x \in \R} \Big| \int_0^x f (y) dy \Big| < \infty,
\quad
\lim_{x \to -\infty} \int_0^x f(y) dy
\text{ exists in } \C.
\]
For further results,
see \cite{HaNa02, St07, Ak25}.
For the quasilinear case, see \cite{KPV04, HGM22, IfTa25, JeOh26, ShZh26} and the references therein.

The aim of this paper is to establish a necessary and sufficient condition on the nonlinearity $F$ for the Cauchy problem \eqref{NLS} to be well-posed in the Sobolev space $H^s(\R)$ for $s>\frac52$.
Our result thus provides a semilinear analogue of the Mizohata condition.

A natural nonlinear analogue of the Mizohata condition is
\begin{equation}
\sup_{x\in\R}
\bigg|
\int_0^x
\Im F_\be
(\phi,\dx\phi,\cj\phi,\cj{\dx\phi})(y) dy
\bigg|
<\infty
\label{Fbez}
\end{equation}
for every $\phi\in H^s(\R)$.
For $s>\frac32$, we have
\[
\phi,\dx\phi\in L^2(\R)\cap L^\infty(\R).
\]
Hence, the quadratic remainder in the Taylor expansion of $F_\be$ at the origin is integrable.
Moreover, the terms linear in $\dx \phi$ and $\cj{\dx \phi}$ have bounded primitives.
Thus,
whether \eqref{Fbez} holds depends only on the constant term and the terms linear in $\phi$ and $\cj \phi$.
Consequently,
\eqref{Fbez} holds for every
$\phi\in H^s(\R)$ if and only if
\[
\Im F_\be(0)=0,
\quad
F_{\cj\al\be}(0)=\cj{F_{\al\be}(0)}.
\]
See Proposition \ref{prop:Miz} for details.
Our main result shows that these algebraic conditions, together with $F(0)=0$,
characterize local well-posedness.

Before stating the main theorem, we introduce the notions of a solution and local well-posedness.
In what follows,
we restrict our attention to solutions of \eqref{NLS} for nonnegative times,
since the case of negative times can be treated similarly.

\begin{definition}
\label{def:sol}
Let $s>\frac32$, $T>0$, and $\phi\in H^s(\R)$.
We say that $u$ is a solution to \eqref{NLS} in $H^s(\R)$
on $[0,T]$ if the following two conditions hold:
\begin{enumerate}
\item[(i)]
$u\in L^\infty([0,T];H^s(\R))$.

\item[(ii)]
For every $\chi\in C_c^\infty([0,T)\times\R)$,%
\footnote{Here, $C_c^\infty([0,T)\times\R)$ denotes the space
of functions that are smooth up to $t=0$ and have compact support
in $[0,T)\times\R$. Such functions need not vanish at $t=0$.}
we have
\begin{align*}
&-\int_0^T\int_\R u(t,x)\dt\chi(t,x) dx dt
-\int_\R\phi(x)\chi(0,x) dx
\\
&\quad
+i\int_0^T\int_\R u(t,x)\dx^2\chi(t,x) dx dt
\\
&=
\int_0^T\int_\R
F(u,\dx u,\cj u,\cj{\dx u})(t,x)\chi(t,x) dx dt.
\end{align*}
\end{enumerate}
\end{definition}

Condition (ii) in Definition \ref{def:sol} incorporates both
the equation and the initial condition in the distributional sense.

\begin{remark}
\label{rem:sold}
\rm

Assume that $F(0)=0$, and let $u$ be a solution in the sense of
Definition \ref{def:sol}.
Lemma \ref{lem:comp} gives
\begin{equation}
F(u,\dx u,\cj u,\cj{\dx u})
\in L^\infty([0,T];H^{s-1}(\R)).
\label{eq:Fnd}
\end{equation}
The equation therefore yields
\[
\dt u\in L^\infty([0,T];H^{s-2}(\R)).
\]
Hence, $u$ has an $H^{s-2}(\R)$-valued absolutely continuous
representative.
Using the weak formulation, including its initial term, we obtain the
Duhamel formula
\begin{equation}
u(t)
=
e^{-it\dx^2}\phi
+
\int_0^t e^{-i(t-t')\dx^2}
F(u,\dx u,\cj u,\cj{\dx u})(t') dt'
\label{eq:Duh}
\end{equation}
as an identity in $H^{s-1}(\R)$ for almost every $t\in[0,T]$.
The integral is understood as a Bochner integral in $H^{s-1}(\R)$.

Throughout this paper, we use the right-hand side of \eqref{eq:Duh} as
the representative of $u$.
By \eqref{eq:Fnd}, this representative belongs to
$C([0,T];H^{s-1}(\R))$.
Together with the $L^\infty([0,T];H^s)$ bound, this implies that it
belongs to $C_w([0,T];H^s(\R))$.
Here, $C_w([0,T];H^s(\R))$ denotes the space of weakly continuous
$H^s(\R)$-valued functions on $[0,T]$.
With this convention, \eqref{eq:Duh} holds for every $t\in[0,T]$ and
$u(0)=\phi$.
\end{remark}

We use the following notion of local well-posedness.

\begin{definition}
\label{def:wp}
For $s>\frac32$, we say that \eqref{NLS} is locally well-posed
in $H^s(\R)$ if, for every $\psi\in H^s(\R)$, there exist
$T>0$ and an open neighborhood $U$ of $\psi$ in $H^s(\R)$
such that, for each $\phi\in U$, there is a unique solution
$u^{[\phi]}$ to \eqref{NLS} on $[0,T]$ with initial datum $\phi$
in the sense of Definition \ref{def:sol}, and the solution map
\[
U\ni\phi\mapsto u^{[\phi]}
\in L^\infty([0,T];H^s(\R))
\]
is continuous.
\end{definition}

\begin{remark}
\rm

Assume that $F(0)=0$ and that \eqref{NLS} is locally well-posed
in $H^s(\R)$.
Let $T$ and $U$ be as in Definition \ref{def:wp}.
For each $\phi\in U$, the representative of $u^{[\phi]}$
specified in Remark \ref{rem:sold} belongs to
$C([0,T];H^s(\R))$.
We give a proof of this fact in Appendix \ref{SEC:ASCT} for completeness.
Moreover, the solution map
\[
U\ni\phi
\mapsto u^{[\phi]}
\in C([0,T];H^s(\R))
\]
is continuous.
\end{remark}

We are now ready to state our main result.

\begin{theorem}
\label{thm:main}
Let $s>\frac52$ and let $F\in C^\infty(\C^2;\C)$.
The following are equivalent.
\begin{enumerate}
\item
The Cauchy problem \eqref{NLS} is locally well-posed in $H^s(\R)$.

\item
There exists a nonempty open ball $B\subset H^s(\R)$ such that,
for every initial datum $\phi\in B$,
there exist $T_\phi>0$ and a solution to \eqref{NLS} in $H^s(\R)$ on $[0,T_\phi]$
in the sense of Definition \ref{def:sol}.

\item
The following three conditions hold:
\begin{equation}
F(0)=0,
\quad
\Im F_\be(0)=0,
\quad
F_{\cj\al\be}(0)=\cj{F_{\al\be}(0)}.
\label{eq:condition}
\end{equation}
\end{enumerate}

If any of the conditions in \eqref{eq:condition} fails, then, for every
$\psi\in H^s(\R)$ and every $\eps>0$, there exists
$\phi\in H^s(\R)$ such that
\[
\|\phi-\psi\|_{H^s}<\eps
\]
and
\eqref{NLS} with initial datum $\phi$ admits no solution
$u\in C_w([0,T];H^s(\R))$ for any $T>0$.
\end{theorem}

Condition (ii) in Theorem \ref{thm:main} requires only the existence of solutions for all initial data in a single nonempty open ball.
In particular, neither uniqueness nor continuous dependence is required,
and the lifespan may depend on the initial datum.
Nevertheless, this weak existence assumption already forces the full structural condition \eqref{eq:condition} on the nonlinearity.

The assumption in condition (ii) cannot, in general, be replaced by existence on a dense subset of $H^s(\R)$.
Indeed, consider
\[
F(\al,\be,\cj\al,\cj\be)=i\be.
\]
Then
\[
F(0)=0,
\quad
\Im F_\be(0)=1,
\]
so \eqref{eq:condition} fails.
On the other hand, if $\ft \phi$ has compact support,
then \eqref{NLS} admits a global solution in $H^s(\R)$,
given by
\[
\widehat{u}(t,\xi)
=
e^{it\xi^2-t\xi}\widehat{\phi}(\xi).
\]
Since such initial data are dense in $H^s(\R)$, existence on a dense
subset alone does not imply the structural condition in
Theorem \ref{thm:main}.

We now outline the proof of Theorem~\ref{thm:main}.
Since (i)$\implies$(ii) is immediate,
it suffices to prove (ii)$\implies$(iii) and (iii)$\implies$(i).

The proof of (iii)$\implies$(i) is based on the energy method.
However, a direct application of this method leads to a loss of derivatives.
To overcome this difficulty, we set
\[
v := \dx u,
\quad
w := \dx^2 u.
\]
Differentiating \eqref{NLS} twice, we obtain
\begin{equation}
\begin{aligned}
\dt w + i \dx^2 w
&=
i (\Im F_\be (u, v, \cj u, \cj v)) \dx w
+ (\Re F_\be (u, v, \cj u, \cj v)) \dx w
\\
&\quad
+ F_{\cj \be}(u, v, \cj u, \cj v) \cj {\dx w}
\\
&\quad
+ \text{(terms involving no derivatives of $w$ or $\cj w$)}.
\end{aligned}
\label{eqw1}
\end{equation}
The first term on the right-hand side causes a loss of derivatives,
whereas the other terms can be handled by the energy method.

To eliminate the troublesome term, we apply a gauge transformation.
For related gauge transformations,
see \cite{HaOz92, Taka99}.
Thanks to \eqref{Fbez}, we can define the bounded real-valued function
\[
\Ld(t,x)
:=
- \frac 12 \int_0^x
\Im F_\be (u, v, \cj u, \cj v)(t,y) dy.
\]
Setting
\[
W := e^\Ld w,
\]
we observe that
\[
F_\be(u,v,\cj u,\cj v)+2i\dx\Ld
=
\Re F_\be(u,v,\cj u,\cj v).
\]
Consequently, the equation for $W$ takes the form
\[
\begin{aligned}
\dt W + i \dx^2 W
&=
(\Re F_\be (u, v, \cj u, \cj v)) \dx W
+ F_{\cj \be}(u, v, \cj u, \cj v) \cj {\dx W}
\\
&\quad
+ (\dt \Ld + i \dx^2 \Ld) W
+ \text{(other terms)}.
\end{aligned}
\]
Although $\dt \Ld$ need not belong to $L^2(\R)$,
the term $(\dt \Ld) W$ can be controlled using \eqref{NLS} and the algebraic conditions in \eqref{eq:condition}. 
See Lemma~\ref{lem:remain}.
Since the coefficient of $\dx W$ is real-valued,
the remaining first-order terms involving $\dx W$ and $\cj{\dx W}$ can be controlled by the energy estimates in Lemma \ref{lem:ened}.

To construct a solution to \eqref{NLS} rigorously,
we introduce a parabolic regularization of \eqref{NLS}.
We then establish uniform energy estimates for the regularized solutions, independent of the regularization parameter.
Passing to the limit, we obtain a solution to \eqref{NLS}.
Finally, continuous dependence on the initial data in $H^s(\R)$ is established by a Bona-Smith-type approximation argument
as in \cite{BoSm75}.

Our energy estimate for $W$ requires $s>\frac52=2+\frac12$,
since we use the Sobolev embedding
$\| W \|_{L^\infty}\les \|W\|_{H^{s-2}}$ for $s-2>\frac12$.
Although this regularity threshold is not optimal
(see, e.g., \cite{Por18} for nonlinearities of cubic or higher order),
we do not pursue lower regularity here,
since our focus is the algebraic characterization of well-posedness.

We next outline the proof of (ii)$\implies$(iii).
We argue by contraposition.
If $F(0)\ne0$, the fact that a nonzero constant function does not belong to $H^s(\R)$ allows us
to show that \eqref{NLS} admits no $H^s(\R)$-solution for any initial datum in $H^s(\R)$.
See Proposition \ref{prop:F0}.

Assume that $F(0)=0$ and that at least one of the other two conditions in \eqref{eq:condition} fails.
Our argument is based on the wave-packet method used to establish the necessity of the Mizohata condition \cite{Mi85};
see also \cite{JeOh26}.
We prove that,
for any fixed existence time and any prescribed bound on the solutions,
the set of initial data admitting solutions with these bounds has empty interior in $H^s(\R)$
(see Theorem~\ref{thm:necessity}).
Combining this with the closedness of these sets and the Baire category theorem,
we obtain the non-existence result for a dense subset of $H^s(\R)$ stated in Theorem \ref{thm:main}.

We now explain the main idea of the wave-packet argument.
The failure of \eqref{eq:condition} allows us to construct
a smooth function $\psi$, arbitrarily close to any given
$\phi\in H^s(\R)$, such that
\[
a_0(x)
:=
F_\be(\psi,\dx\psi,\cj\psi,\cj{\dx\psi})
\]
does not satisfy the Mizohata condition.
In view of \eqref{eqw1}, we consider the model equation
\[
\dt w=(-i\dx^2+a_0\dx)w.
\]
For a smooth function $h$, we formally have
\begin{equation}
\frac d{dt}(w,h)_{L^2}
=
(w,(\dt+(-i\dx^2+a_0\dx)^\ast )h)_{L^2}
+\text{(error terms)},
\label{wp1a}
\end{equation}
where $(-i\dx^2+a_0\dx)^\ast $ denotes the formal adjoint of $-i\dx^2+a_0\dx$ in $L^2(\R)$.
We therefore construct a high-frequency wave packet $h$ satisfying
\[
\dt h+(-i\dx^2+a_0\dx)^\ast h
=\text{(error)}.
\]
By exploiting the failure of the Mizohata condition,
we can choose $h$ so that its $H^{-(s-2)}(\R)$-norm becomes small as it propagates through a suitable spatial interval.

More precisely, assuming the existence of solutions with a uniform $H^s(\mathbb R)$ bound,
we construct a sequence of wave packets $h_N$ for $N \in \N$,
corresponding times $T_N \to 0$ as $N \to \infty$,
and solutions $u_N$ such that $\|u_N(0)-\psi\|_{H^s}\ll1$ and
\[
|(w_N(0),h_N(0))_{L^2}|\ge c
\]
for some $c>0$ independent of $N$.
At time $T_N$, however,
we obtain
\[
|(w_N(T_N),h_N(T_N))_{L^2}|\to0.
\]

To derive a contradiction,
we estimate
\[
\int_0^{T_N}
\frac d{dt}(w_N(t),h_N(t))_{L^2} dt.
\]
The main difficulty comes from the third term on the right-hand side of \eqref{eqw1},
which involves $\cj{\dx w}$.
Unlike the term $a_0\dx w$, it carries the opposite Schr\"odinger phase.
Exploiting this nonresonance,
we integrate its contribution to the dual pairing by parts in time and gain a decaying factor at high frequency.
The other error terms are controlled by the uniform $H^s$ bounds and the time regularity of the solutions.
Consequently, integrating \eqref{wp1a}, we obtain
\begin{align*}
&|(w_N(T_N),h_N(T_N))_{L^2}-(w_N(0),h_N(0))_{L^2}|
\\
&=
\bigg|\int_0^{T_N}
\frac d{dt} (w_N(t),h_N(t))_{L^2} dt\bigg|
\to0,
\end{align*}
which contradicts the preceding two estimates.

\begin{remark}
\rm

On the torus $\T:=\R/(2\pi\Z)$,
the authors proved in \cite{KoOk25} that,
for polynomial nonlinearities,
the Cauchy problem is locally well-posed in $H^s(\T)$
for $s>\frac52$ if and only if
\[
\int_\T
\Im F_\be(\phi,\dx \phi,\cj \phi,\cj{\dx \phi}) dx=0
\]
for every $\phi \in H^s(\T)$.
As explained after Theorem 1.4 in \cite{KoOk25},
this condition is equivalent to the existence of
a smooth real-valued function $\Phi(\al,\cj\al)$ such that
\[
\Im F_\be(\phi,\dx \phi,\cj \phi,\cj{\dx \phi})
=\dx ( \Phi(\phi,\cj \phi) )
\]
for every $\phi\in H^s(\T)$.

This contrasts with the condition \eqref{Fbez} on $\R$.
Indeed, if $F$ is a polynomial consisting only of terms of degree at least three,
then \eqref{Fbez} holds automatically,
since $\Im F_\be(\phi,\dx \phi,\cj \phi,\cj{\dx \phi})$
belongs to $L^1(\R)$ for every $\phi \in H^s(\R)$.
On $\T$, however, the integral condition may fail
even for such nonlinearities.
\end{remark}

We briefly describe the organization of the paper.
In Section \ref{sec:prel},
we collect the analytic estimates used throughout the paper.
In Section \ref{sec:miz},
we characterize the nonlinear Mizohata condition \eqref{Fbez} in terms of the derivatives of $F$ at the origin.
In Section \ref{sec:suff},
we prove the sufficiency of \eqref{eq:condition}.
In Section \ref{sec:necessity},
we prove the necessity  of \eqref{eq:condition}.
Finally,
in Appendix \ref{SEC:ASCT},
we prove the strong continuity in time of the solutions appearing in Definition \ref{def:wp}.

\subsection{Notation}

We conclude this section by fixing notation.
For a Banach space $X$, we denote by $B_X(\phi,\rho)$ the ball centered
at $\phi\in X$ with radius $\rho>0$:
\[
B_X(\phi,\rho)
:=
\{f\in X\mid \|f-\phi\|_X<\rho\}.
\]

For $T>0$,
let $L^\infty([0,T];X)$ be the space of essentially bounded $X$-valued functions on $[0,T]$
equipped with the norm
\[
\|u\|_{L_T^\infty X}
:=
\esssup_{t\in[0,T]}
\|u(t)\|_{X}.
\]
Let $C_w([0,T];X)$ be the space of weakly continuous $X$-valued functions on $[0,T]$.
If $u\in C_w([0,T];X)$, then
$t\mapsto\|u(t)\|_{X}$ is lower semicontinuous on $[0,T]$.
Consequently,
we have
\begin{equation}
\sup_{t\in[0,T]}\|u(t)\|_{X}
=
\|u\|_{L_T^\infty X}.
\label{NNN1}
\end{equation}
Let $AC([0,T];X)$ be the space of absolutely continuous $X$-valued functions on $[0,T]$.

For $k \in \N$ and $p \in [1,\infty]$,
we denote by $W^{k,p}(\R)$
the Sobolev space equipped with the norm
\[
\|f\|_{W^{k,p}}
:=
\| f \|_{L^p} + \| \dx^k f \|_{L^p}.
\]

When $f \in L^1(\R)$, we define
\[
\ft f(\xi)
:= \frac 1{\sqrt {2\pi}} \int_\R e^{-ix \xi} f(x) dx.
\]
We denote by $\ft f$ or $\Ft [f]$ the Fourier transform of a tempered distribution $f$.
For $s \in \R$,
let $J^s$ be the Fourier multiplier with symbol $\jb{\xi}^s$,
where
\[
\jb{\xi} := (1+|\xi|^2)^{\frac 12}.
\]
We denote by $H^s(\R)$ the $L^2$-based Sobolev space equipped with the norm
\[
\|f\|_{H^s}
:=
\| J^s f \|_{L^2}
= \|\jb{\cdot }^s \ft f\|_{L^2}.
\]
We write
\[
(f,g)_{L^2}:=
\int_\R f \cj{g} dx,
\quad
(f,g)_{H^s}
:= (J^s f,J^s g)_{L^2}.
\]

We write $A \les B$ if $A \le CB$ for some constant $C>0$.
When the constant $C>0$ is sufficiently small,
we write $A \ll B$.
We write $A \sim B$ if both $A \les B$ and $B \les A$ hold.

\section{Preliminaries}
\label{sec:prel}

We state the analytic results in the precise forms used below.

\begin{lemma}
\label{lem:product}

\begin{enumerate}
\item
For $r>0$ and $f,g\in H^r (\R) \cap L^\infty (\R)$,
\[
\|fg\|_{H^r}
\les
\|f\|_{L^\infty} \|g\|_{H^r}
+\|g\|_{L^\infty} \|f\|_{H^r}
.
\]

\item
For $\s> \frac 12$ and $|r|\le \s$,
multiplication extends to a continuous bilinear map
$H^\s(\R) \times H^r(\R) \to H^r(\R)$, with
\[
\|fg\|_{H^r}
\les
\|f\|_{H^\s}\|g\|_{H^r}.
\]
In particular, $H^\s(\R)$ is an algebra, and
\[
\|fg\|_{H^{\s-1}}
\les
\|f\|_{H^\s} \|g\|_{H^{\s-1}}
\]
for $\s>\frac12$.
\end{enumerate}

\end{lemma}

\begin{proof}
For (i),
see Lemma X4 in \cite{KaPo88}.
For (ii),
see Corollary 3.16 on page 855 in \cite{Tao01}.
\end{proof}

\begin{lemma}
\label{lem:comp}
Let $F \in C^\infty(\C^2;\C)$.
For $r>1$ and $M>0$,
there exists $C=C(r,F,M)>0$ such that
\begin{align}
&\| F(f,\dx f,\cj f,\cj{\dx f}) - F(0)\|_{H^{r-1}}
\le C \|f\|_{H^r},
\label{comp1}
\\
\begin{split}
&\|F(f,\dx f,\cj{f},\cj{\dx f})-F(g,\dx g,\cj{g},\cj{\dx g})\|_{H^{r-1}}
\\
&\le C
\{
\|f-g\|_{H^r}
+
(\|f\|_{H^r}+\|g\|_{H^r})
\|f-g\|_{W^{1,\infty}}
\}
\end{split}
\label{comp2}
\end{align}
for any
$f,g \in H^r (\R) \cap W^{1,\infty}(\R)$ with
\[
\| f \|_{W^{1,\infty}} + \| g \|_{W^{1,\infty}} \le M.
\]
\end{lemma}

\begin{proof}
For the composition estimate \eqref{comp1},
see Proposition 3.1 in \cite{ChWe91}.
See also Theorem 1 on page 367 and Theorem 2 on page 368 in \cite{RuSi96}.
For the difference estimate \eqref{comp2},
we write
$F(f,\dx f,\cj f,\cj{\dx f})=F[f]$ for short.
We have
\begin{align*}
F[f]-F[g]
&=
\int_0^1
\{
F_\al [f + \theta (f-g)] (f-g)
+
F_\be [f + \theta (f-g)] \dx (f-g) 
\\
&\quad
+
F_{\cj \al} [f + \theta (f-g)] (\cj f-\cj g) 
+
F_{\cj \be} [f + \theta (f-g)] \dx (\cj f-\cj g) 
\} d\theta.
\end{align*}
Then,
the desired estimate follows from Lemma \ref{lem:product} and \eqref{comp1}.
\end{proof}

\begin{corollary}
\label{cor:comppro}
Let $\s > \frac 12$, $r>1$, and $F \in C^\infty(\C^2;\C)$.
Then, we have
\begin{equation}
\begin{aligned}
&\| F(f, \dx f, \cj f, \cj {\dx f}) g \|_{H^{r-1}}
\le
C(\|f\|_{H^{\s+1}}, \|g\|_{H^{\s}})
\big( \|g\|_{H^{r-1}} + \|f\|_{H^r} \big)
\end{aligned}
\label{compro}
\end{equation}
for any
\[
f\in H^{\s+1}(\R)\cap H^r(\R),
\quad
g\in H^{\s}(\R)\cap H^{r-1}(\R).
\]
\end{corollary}

\begin{proof}
In this proof, we write
$F(f,\dx f,\cj f,\cj{\dx f})=F[f]$ for short.
Since $\s>\frac12$,
the Sobolev embedding $H^\s(\R) \hookrightarrow L^\infty(\R)$ and Lemma \ref{lem:comp} give
\[
\|F[f]-F[0]\|_{L^\infty}
\le C(\|f\|_{H^{\s+1}}),
\quad
\|g\|_{L^\infty}\les\|g\|_{H^\s}.
\]
Moreover, since $r-1>0$, Lemma \ref{lem:comp} yields that
\[
\begin{aligned}
\|F[f]-F[0]\|_{H^{r-1}}
&\le C(\|f\|_{H^{\s+1}})
\big(\|f\|_{H^{r-1}}+\|\dx f\|_{H^{r-1}}\big)
\\
&\le C(\|f\|_{H^{\s+1}})\|f\|_{H^r}.
\end{aligned}
\]
Hence, Lemma \ref{lem:product} (i) implies
\begin{align*}
\|F[f]g\|_{H^{r-1}}
&\les
\|g\|_{H^{r-1}}
+\|(F[f]-F[0])g\|_{H^{r-1}}
\\
&\les
\|g\|_{H^{r-1}}
+\|F[f]-F[0]\|_{H^{r-1}}\|g\|_{L^\infty}
\\
&\quad
+\|F[f]-F[0]\|_{L^\infty}\|g\|_{H^{r-1}}
\\
&\le
C(\|f\|_{H^{\s+1}})
\|g\|_{H^{r-1}}
+
C(\|f\|_{H^{\s+1}},\|g\|_{H^\s})
\|f\|_{H^r}.
\end{align*}
This proves \eqref{compro}.
\end{proof}

\begin{lemma}
\label{lem:prod2}
Let $r \ge 0$ and
\[
0<
\dl<
\min\Big( \frac12,1-r+[r]\Big),
\]
where $[r]$ means the largest integer less than or equal to $r$.
Then,
we have
\[
\| f g \|_{H^r} \les \|f \|_{H^r} \| g \|_{L^\infty}
+ \| f \|_{H^{r-[r]-\frac 12 + \dl}} \| \dx^{[r]+1} g \|_{L^2}
\]
for
$f \in H^r(\R)$ and
$g \in L^\infty(\R)$ with $\dx^{[r]+1} g \in L^2(\R)$.
\end{lemma}

\begin{proof}
Let $P_N$ be the inhomogeneous Littlewood-Paley projector onto frequencies
$|\xi|\sim N$ when $N \in 2^\N$ and $|\xi| \le 1$ when $N=1$.
Note that
\[
f = \sum_{N \in 2^{\N_0} } P_N f
\]
holds for any tempered distribution.

Using the paraproduct decomposition, we have
\begin{equation}
\begin{aligned}
f g
=
\sum_{ \substack {N_1, N_2 \in 2^{\N_0} \\ N_1 \gg N_2} }
(P_{N_1} f) P_{N_2} g
+
\sum_{ \substack {N_1, N_2 \in 2^{\N_0} \\ N_1 \les N_2} }
(P_{N_1} f) P_{N_2} g
=:
\I + \II.
\label{pro1}
\end{aligned}
\end{equation}
For $\I$, put
\[
P_{\ll N}g:=\sum_{\substack{N_2 \in 2^{\N_0} \\ N_2\ll N}} P_{N_2}g.
\]
Since $(P_Nf)P_{\ll N}g$ is localized to frequencies $|\xi|\sim N$ and the kernel of $P_{\ll N}$ has uniformly bounded $L^1$-norm, almost orthogonality gives
\begin{align*}
\|\I\|_{H^r}^2
&\les
\sum_{N\in2^{\N_0}}N^{2r}\| (P_Nf) P_{\ll N}g\|_{L^2}^2
\\
&\les
\|g\|_{L^\infty}^2\sum_{N\in2^{\N_0}} N^{2r} \|P_Nf\|_{L^2}^2
\les
\|f\|_{H^r}^2\|g\|_{L^\infty}^2.
\end{align*}
It remains to estimate $\II$ in \eqref{pro1}.
Since the Fourier support of $(P_{N_1}f) P_{N_2}g$ is contained in $\{|\xi|\lesssim N_2\}$ when $N_1\les N_2$, we have
\[
\|\II\|_{H^r}
\les
\sum_{\substack{N_1,N_2\in2^{\N_0}\\N_1\les N_2}}
N_2^r \| (P_{N_1}f) P_{N_2}g\|_{L^2}.
\]
By Bernstein's inequality, for $N_2>1$ and $N_1\les N_2$,
\begin{align*}
&N_2^r \|(P_{N_1}f) P_{N_2}g\|_{L^2}
\\
&\les
N_2^r N_1^{\frac12}
\|P_{N_1}f\|_{L^2}\|P_{N_2}g\|_{L^2}
\\
&\les
\Big(\frac{N_1}{N_2}\Big)^{1-r+[r]-\dl} N_2^{-\dl}
\|P_{N_1} N_1^{r-[r]-\frac12+\dl}f\|_{L^2}
\|P_{N_2} N_2^{[r]+1}g\|_{L^2}.
\end{align*}
The Cauchy-Schwarz inequality gives
\begin{align*}
&\sum_{\substack{N_2\in2^\N, \, N_1\in2^{\N_0}\\N_1\les N_2}}
N_2^r \|P_{N_1}fP_{N_2}g\|_{L^2}
\les
\|f\|_{H^{r-[r]-\frac12+\dl}}
 \| \dx^{[r]+1} g \|_{L^2}
.
\end{align*}
When $N_2\les1$, both frequencies are bounded and
\[
\sum_{\substack{N_1,N_2\in2^{\N_0} \\ N_1\les N_2\les1}}
N_2^r \|P_{N_1}fP_{N_2}g\|_{L^2}
\les
\|f\|_{L^2}\|g\|_{L^\infty}.
\]
Consequently,
\[
\|\II\|_{H^r}
\les
\|f\|_{L^2}\|g\|_{L^\infty}
+\|f\|_{H^{r-[r]-\frac12+\dl}}
\| \dx^{[r]+1} g \|_{L^2}
.
\]
Together with \eqref{pro1}, this proves the desired bound.
\end{proof}

For the gauge transformation,
we use the following bound.

\begin{lemma}
\label{lem:gauge}
Let $\s, r \in \R$ satisfy
\[
\frac 12 < \s \le r.
\]
Let $M>0$.
Assume that $\Ld \in L^\infty(\R)$ satisfies $\dx \Ld \in H^r (\R)$ and
\[
\| \Ld \|_{L^\infty} + \| \dx \Ld \|_{H^\s} \le M.
\]
Then,
there exists a constant $C(M)>0$ such that
\[
\| e^\Ld f\|_{H^r}
\le
C(M)
(\|f\|_{H^r}
+ \|\dx \Ld \|_{H^r}\|f\|_{H^\s})
\]
for any $f\in H^r(\R)$.
\end{lemma}

\begin{proof}
It suffices to show that
\begin{equation}
\| \dx^{[r]+1} e^\Ld \|_{L^2}
\le
C(M)
\| \dx \Ld \|_{H^r}.
\label{ggg0}
\end{equation}
Indeed, once \eqref{ggg0} is established,
Lemma \ref{lem:prod2} gives
\begin{align*}
\| e^\Ld f\|_{H^r}
&\les
\| f \|_{H^r} \| e^\Ld \|_{L^\infty}
+ \| f \|_{H^\s} \| \dx^{[r]+1} e^\Ld \|_{L^2}
\\
&\le
C(M)
( \| f \|_{H^r}
+ \| \dx \Ld \|_{H^r} \| f \|_{H^\s}).
\end{align*}

A direct calculation yields that
\begin{equation}
\| \dx^{[r]+1} e^\Ld \|_{L^2}
\les
\| e^{\Ld} \|_{L^\infty}
\sum_{k=1}^{[r]+1}
\sum_{\substack{j_1, \dots, j_k \in \N \\
j_1+ \dots+ j_k = [r]+1}}
\bigg\|
\prod_{\l=1}^k \dx^{j_\l} \Ld \bigg\|_{L^2}
.
\label{ggga}
\end{equation}
When $k=[r]+1$,
H\"older's inequality and the Sobolev embedding $H^\s(\R) \hookrightarrow L^\infty(\R)$ imply that
\begin{equation}
\bigg\|
\prod_{\l=1}^{[r]+1} \dx \Ld \bigg\|_{L^2}
\le
\| \dx \Ld \|_{L^\infty}^{[r]}
\| \dx \Ld \|_{L^2}
\les
\| \dx \Ld \|_{H^\s}^{[r]+1}
\le C(M)
\| \dx \Ld \|_{H^r}
.
\label{gggb}
\end{equation}
In what follows, we consider the case $k<[r]+1$.
For each term in the sum, set
\[
m_\l :=j_\l-1
\]
for $\l=1,\dots, k$.
Note that
\[
m_\l \ge0,
\quad
m_1+\cdots+m_k
=[r]+1-k \ge 1.
\]
Let
\[
\mm _k := (m_1, \dots, m_k).
\]
We set
\begin{align*}
L_0 (\mm_k) &:= \{ \l \in \{ 1, \dots, k \} \mid
m_\l =0 \},
\\
L_1 (\mm_k) &:= \{ \l \in \{ 1, \dots, k \} \mid
m_\l >0 \}.
\end{align*}

For $\l \in L_1(\mm_k)$,
the Gagliardo-Nirenberg inequality yields that
\begin{equation}
\begin{aligned}
\|\dx^{m_\l} \dx \Ld \|_{L^{\frac{2([r]+1-k)}{m_\l}}}
&\les
\| \dx \Ld \|_{L^\infty}^{1-\frac{m_\l}{[r]+1-k}}
\| \dx^{[r]+1-k} \dx \Ld \|_{L^2}^{\frac{m_\l}{[r]+1-k}}
\\
&\le
C(M)
\| \dx \Ld \|_{H^r}^{\frac{m_\l}{[r]+1-k}}.
\end{aligned}
\label{ggg1}
\end{equation}
It follows from H\"older's inequality and \eqref{ggg1}
that
\begin{equation}
\begin{aligned}
&\bigg\|
\prod_{\l=1}^k \dx^{m_\l} \dx \Ld
\bigg\|_{L^2}
\\
&\le
\bigg(
\prod_{\l \in L_0(\mm_k)} \|\dx \Ld \|_{L^\infty}
\bigg)
\bigg(
\prod_{\l \in L_1(\mm_k)}
\|\dx^{m_\l} \dx \Ld \|_{L^{\frac{2([r]+1-k)}{m_\l}}}
\bigg)
\\
&\le
C(M)
\prod_{\l \in L_1 (\mm_k)}
\| \dx \Ld \|_{H^r}^{\frac{m_\l}{[r]+1-k}}
\le
C(M)
\| \dx \Ld \|_{H^r}.
\end{aligned}
\label{gggc}
\end{equation}
From \eqref{ggga}, \eqref{gggb}, and \eqref{gggc},
we have \eqref{ggg0}.
\end{proof}

To prove local well-posedness in $H^s(\R)$ for $s>\frac 52$,
we use the following commutator estimate.
Set 
\[
[J^r, f] \dx g := J^r (f \dx g) - f J^r \dx g.
\]

\begin{lemma}
\label{lem:Com1}
Let $r \ge 0$.
Then,
we have
\begin{align*}
\| [J^r, f] \dx g \|_{L^2}
&\les
\| f \|_{H^{\frac 32+\dl}} \| g \|_{H^r}
+ \| f \|_{H^{r+1}} \| g \|_{H^{\frac 12+\dl}}
\end{align*}
for any $\dl>0$, 
$f \in H^{\max (r+1, \frac 32 + \dl)}(\R)$, 
and $g \in H^{\max (r, \frac 12 + \dl)}(\R)$.
\end{lemma}

The corresponding estimate on $\T$ was proved in Proposition 2.4 of \cite{KoOk25}.
For completeness, we give a proof here.

\begin{proof}
The case $r=0$ is trivial.
Assume $r>0$.
By density, it suffices to consider Schwartz functions.
We write
\begin{equation}
\Ft \big[ [J^r, f] \dx g \big] (\xi)
=
\frac 1{\sqrt{2\pi}}
\int_\R (\jb{\xi}^r - \jb{\eta}^r ) i \eta \ft f (\xi - \eta) \ft g(\eta) d \eta.
\label{Com1a}
\end{equation}
We first note that
\begin{equation}
|\jb{\xi}^r-\jb{\eta}^r| |\eta|
\les
\jb{\xi-\eta}\jb{\eta}^r
+
\jb{\xi-\eta}^{r+1}
\label{Com1b}
\end{equation}
Indeed,
if
$|\xi-\eta|\ll\jb{\eta}$,
then $\jb{\xi}\sim\jb{\eta}$, and the mean value theorem gives
\[
|\jb{\xi}^r-\jb{\eta}^r|
\les
|\xi-\eta|\jb{\eta}^{r-1},
\]
which yields the first term in \eqref{Com1b}.
If
$|\xi-\eta| \ges \jb{\eta},$
then
$\jb{\xi} \les \jb{\xi-\eta}$
and the second term on the right-hand side of \eqref{Com1b} gives the desired bound.

By Plancherel's theorem, \eqref{Com1a}, and \eqref{Com1b},
we have
\begin{align*}
\| [J^r,f]\dx g \|_{L^2}
\les
\|(\jb{\cdot}^{r+1}|\ft f|)\ast|\ft g|\|_{L^2}
+
\|(\jb{\cdot}|\ft f|) \ast(\jb{\cdot}^r|\ft g|)\|_{L^2}.
\end{align*}
Young's inequality gives
\begin{align*}
\| [J^r,f]\dx g \|_{L^2}
&\les
\|\jb{\cdot}\ft f\|_{L^1}
\| g \|_{H^r}
+
\| f \|_{H^{r+1}} \| \ft g\|_{L^1}
\\
&\les
\|f\|_{H^{\frac32+\dl}}\|g\|_{H^r}
+
\|f\|_{H^{r+1}}\|g\|_{H^{\frac12+\dl}},
\end{align*}
which proves the assertion.
\end{proof}

For later use, we record the following compactness lemma.

\begin{lemma}
\label{lem:cpt}
Let $s\in\R$, $T>0$, and
\[
\{u_n\}_{n\in\N}
\subset
C_w([0,T];H^s(\R))
\cap
AC([0,T];H^{s-2}(\R)).
\]
Assume that
\[
M_0
:=
\sup_{n\in\N}
\|u_n\|_{L_T^\infty H^s}
<\infty,
\quad
M_1
:=
\sup_{n\in\N}
\|\dt u_n\|_{L_T^\infty H^{s-2}}
<\infty.
\]
Then, there exist a subsequence, still denoted by
$\{u_n\}_{n\in\N}$, and
\[
u
\in
C_w([0,T];H^s(\R))
\cap
AC([0,T];H^{s-2}(\R))
\]
such that
\begin{align*}
u_n
& \wsl 
u
\quad\text{in }
L^\infty([0,T];H^s(\R)),
\\
\dt u_n
& \wsl 
\dt u
\quad\text{in }
L^\infty([0,T];H^{s-2}(\R)),
\end{align*}
and
\[
\sup_{t\in[0,T]}
\|u(t)\|_{H^s}
\le M_0,
\quad
\|\dt u\|_{L_T^\infty H^{s-2}}
\le M_1.
\]
Moreover,
if $s>\frac 32$,
then
\[
u_n\to u,
\quad
\dx u_n\to\dx u
\]
locally uniformly on $[0,T]\times\R$.

Furthermore, if
\[
u_n(0)\to\phi
\]
in the sense of distributions on $\R$, then
$u(0)=\phi$.
\end{lemma}

\begin{proof}
By the Banach-Alaoglu theorem, after taking a subsequence,
there exist
\[
u\in L^\infty([0,T];H^s(\R)),
\quad
v\in L^\infty([0,T];H^{s-2}(\R))
\]
such that
\[
u_n \wsl u
\text{ in }
L^\infty([0,T];H^s(\R)),
\quad
\dt u_n \wsl v
\text{ in }
L^\infty([0,T];H^{s-2}(\R)).
\]
Passing to the limit in the definition of the distributional time derivative,
we obtain
\[
v=\dt u.
\]
Hence,
$u\in AC([0,T];H^{s-2}(\R))$,
and
\[
\|u\|_{L_T^\infty H^s}\le M_0,
\quad
\|\dt u\|_{L_T^\infty H^{s-2}}\le M_1.
\]
The $H^{s-2}(\R)$-continuous representative of $u$ belongs to
$C_w([0,T];H^s(\R))$ and
\eqref{NNN1} gives
\[
\sup_{t\in[0,T]}\|u(t)\|_{H^s}\le M_0.
\]

Fix $0<\dl<2$.
For $k \in \N$, let $\chi_k \in C_c^\infty(\R)$ satisfy $\chi_k(x) =1$ for $|x| \le k$.
Since
\[
\supp (\chi_k u_n)
\subset \supp \chi_k,
\quad
\sup_{n \in \N} (\| \chi_k u_n \|_{L_T^\infty H^s}
+ \| \dt( \chi_k u_n) \|_{L_T^\infty H^{s-2}}) < \infty,
\]
the compactness theorem \cite[Corollary 4, p.85]{Simon87} yields,
after taking a subsequence,
that there exists $U_k \in C([0,T];H^{s-\dl}(\R))$ satisfying
\[
\chi_k u_n\to U_k
\quad\text{in }C([0,T];H^{s-\dl}(\R)).
\]
The weak-$\ast$ convergence implies $U_k=\chi_k u$.
A diagonal argument yields,
after taking a further subsequence,
that
\begin{equation}
\chi u_n\to\chi u
\quad\text{in }
C([0,T];H^{s-\dl}(\R))
\label{eq:cpt01}
\end{equation}
for every
$\chi\in C_c^\infty(\R)$.

If $s>\frac 32$,
we can choose $\dl$ satisfying $0<\dl< s-\frac 32$.
Then,
Sobolev embedding gives
\[
H^{s-\dl}(\R)\subset W^{1,\infty}(\R),
\]
and hence
\[
u_n\to u,
\quad
\dx u_n\to\dx u
\]
locally uniformly on $[0,T]\times\R$.

Finally, if $u_n(0)\to\phi$ in the sense of distributions, then
\eqref{eq:cpt01} at $t=0$ implies
\[
\chi u(0)=\chi\phi
\]
for every $\chi\in C_c^\infty(\R)$.
Thus, $u(0)=\phi$.
\end{proof}

\section{Characterization of the Mizohata condition}
\label{sec:miz}

In this section,
we characterize the Mizohata condition \eqref{Fbez} in terms of the derivatives of $F$ at the origin.
For $s>\frac 32$, let
\[
\mathcal G_s
:=
\bigg\{
\phi \in H^s(\R)
\mid
\sup_{x\in\R}
\bigg|
\int_0^x
\Im F_\be(\phi,\dx \phi,\cj \phi,\cj{\dx \phi})(y) dy
\bigg|
<\infty
\bigg\}
.
\]

\begin{proposition}
\label{prop:Miz}
Let $s>\frac 32$ and $F\in C^3(\C^2;\C)$.  The following are equivalent.

\begin{enumerate}
\item
$\Gl_s=H^s(\R)$.

\item
$\Gl_s$ has nonempty interior in $H^s(\R)$.

\item
The following two conditions hold.
\[
\Im F_\be(0)=0,
\quad
F_{\cj\al\be}(0)=\cj{F_{\al\be}(0)}.
\]
\end{enumerate}
When these conditions hold, define the real-valued functions
\begin{align}
\Phi(\al,\cj\al)
&:=
\Im ( F_{\be\be}(0)\al +F_{\cj\be\be}(0)\cj\al ),
\label{MI2}
\\
\begin{split}
G(\al,\be,\cj\al,\cj\be)
&:=\Im F_\be(\al,\be,\cj\al,\cj\be)
-
\Im ( F_{\be\be}(0)\be + F_{\cj\be\be} (0) \cj \be ).
\end{split}
\label{MI3}
\end{align}
Then,
we have
\begin{align}
&\Im F_\be
(\phi,\dx \phi,\cj \phi,\cj{\dx \phi})
=\dx \Phi(\phi,\cj \phi)
+ G (\phi,\dx \phi,\cj \phi,\cj{\dx \phi}),
\label{MI4}
\\
&
\| \Phi(\phi, \cj \phi) \|_{L^\infty}
+
\| G(\phi,\dx \phi,\cj \phi,\cj{\dx \phi}) \|_{L^1}
\le C(\| \phi \|_{H^s})
\label{MI5}
\end{align}
for $\phi \in H^s(\R)$.
\end{proposition}

\begin{proof}

Define
\[
\begin{aligned}
R(\al,\be,\cj\al,\cj\be)
&:=
\Im F_\be(\al,\be,\cj\al,\cj\be)
-\Im F_\be(0)
\\
&\quad
-\Im (F_{\al\be}(0)\al+F_{\cj\al\be}(0)\cj\al)
\\
&\quad
-\Im (F_{\be\be}(0)\be+F_{\cj\be\be}(0)\cj\be).
\end{aligned}
\]
Then,
by definition,
\begin{equation}
\begin{aligned}
\Im F_\be(\al,\be,\cj\al,\cj\be)
&=
\Im F_\be(0)
+\Im (F_{\al\be}(0)\al+F_{\cj\al\be}(0)\cj\al)
\\
&\quad
+\Im (F_{\be\be}(0)\be+F_{\cj\be\be}(0)\cj\be)
+R(\al,\be,\cj\al,\cj\be).
\end{aligned}
\label{MI6}
\end{equation}
Notice that
\[
R(0)=0,
\quad
dR(0)=0,
\]
where
\[
dR:=(R_\al,R_\be,R_{\cj\al},R_{\cj\be}).
\]
Fix $M>0$.
Since $F\in C^3(\C^2;\C)$,
all second-order derivatives of $R$ are bounded on
$\{X\in\C^2:|X|\le M \}$.
By the mean value theorem,
there exists $C>0$ such that
\begin{equation}
|R(X)|\le C|X|^2,
\quad
|dR(X)|\le C|X|
\label{MI7}
\end{equation}
for $X\in\C^2$ with $|X|\le M$.

Assume (iii).
Then, $\Im F_\be (0)=0$ and
\[
\Im (F_{\al\be}(0)\al +F_{\cj\al\be}(0)\cj\al )
= \Im ( (F_{\al\be}(0) - \cj{F_{\cj\al\be}(0)}) \al )
= 0
\]
for every $\al \in \C$.
Hence,
the function $G$ in \eqref{MI3} is precisely the remainder $R$ in \eqref{MI6}.
With \eqref{MI2},
we have
\[
\dx\Phi( \phi,\cj \phi)
=
\Im (F_{\be\be}(0)\dx \phi + F_{\cj\be\be}(0)\cj{\dx \phi} ),
\]
which gives \eqref{MI4}.
Since $s>\frac 32$,
$\phi, \dx \phi \in L^2 (\R) \cap L^\infty(\R)$
for $\phi \in H^s(\R)$.
Thus,
\eqref{MI7}
with $M = 2(\| \phi \|_{L^\infty} + \| \dx \phi \|_{L^\infty} +1)$
implies \eqref{MI5}.

Integrating
\eqref{MI4} from $0$ to $x$, we obtain $\phi \in \Gl_s$.
This proves (iii)$\implies$(i).

Clearly, (i)$\implies$(ii).
It remains to prove (ii)$\implies$(iii).  Suppose that, for some
$\phi\in H^s(\R)$ and $\rho>0$,
\[
B_{H^s}(\phi,\rho)\subset\Gl_s.
\]
Since $s>\frac 32$, both $\phi$ and $\dx\phi$ are continuous and satisfy
\[
\lim_{|x| \to \infty} \phi(x)
= \lim_{|x| \to \infty} \dx \phi(x)
= 0.
\]
Therefore,
\[
\lim_{|x| \to \infty}
\Im F_\be (\phi,\dx\phi,\cj\phi,\cj{\dx\phi}) (x)
=
\Im F_\be (0).
\]
Since $\phi\in\Gl_s$,
the primitive in \eqref{Fbez} is bounded, which forces $\Im F_\be(0)=0$.

Set
\[
L(\al)
:=
\Im (F_{\al\be}(0) \al +F_{\cj\al\be}(0)\cj \al)
= \Im ( (F_{\al\be}(0) - \cj{F_{\cj\al\be}(0)} ) \al).
\]
Suppose that $L\not\equiv0$.
Choose $\al_\ast  \in \C$ with $L(\al_\ast ) \ne 0$
and put
\[
e(x) := \jb{x}^{-1}.
\]
Then,
$e \in H^k(\R)$ for every $k \in \N$,
whereas $e \notin L^1(\R)$.
By choosing a sufficiently small constant $\eta>0$,
we have
\[
\phi_\eta:=\phi+\eta \al_\ast e
\in B_{H^s}(\phi,\rho).
\]
Set
\[
X[\psi] := (\psi,\dx\psi,\cj\psi,\cj{\dx\psi}).
\]
We then obtain from \eqref{MI6}
\begin{equation}
\begin{aligned}
&\Im F_\be (X[\phi_\eta]) - \Im F_\be (X[\phi])
\\
&=\eta L(\al_\ast )e +
\eta\dx\Im (F_{\be\be}(0) \al_\ast  e +F_{\cj\be\be}(0)\cj{\al_\ast } e )
\\
&\quad
+
(R(X[\phi_\eta])-R(X[\phi])).
\end{aligned}
\label{MI8}
\end{equation}
By \eqref{MI7} with $M = 2(\| X[\phi] \|_{L^\infty} + \| X[\phi_\eta] \|_{L^\infty} + 1)$
and the mean value theorem,
we have
\[
|R(X[\phi_\eta])-R(X[\phi])|
\les
(|X[\phi]|+|X[\phi_\eta]|) e.
\]
The right-hand side belongs to $L^1(\R)$ because
\[
X[\phi], \, X[\phi_\eta] \in (L^2 (\R) \cap L^\infty(\R))^4.
\]
Thus,
the second and third terms in
\eqref{MI8} have bounded primitives, while the
first does not.  Since $\phi\in\Gl_s$, it follows that
$\phi_\eta\notin\Gl_s$, a contradiction.
Hence, $L\equiv0$, which is equivalent to $F_{\cj\al\be}(0)=\cj{F_{\al\be}(0)}$.
This proves (iii).
\end{proof}

\section{Well-posedness}
\label{sec:suff}

In this section,
we prove local well-posedness for \eqref{NLS} under the condition \eqref{eq:condition}.

\subsection{Existence and uniqueness}
\label{SUBSEC:energy}

To construct a solution to \eqref{NLS},
we consider the Cauchy problem for the following parabolic regularization of \eqref{NLS}:
\begin{equation}
\left\{
\begin{aligned}
&\dt u^{\eps}+(i-\eps)\dx^2u^{\eps}
=F(u^{\eps},\dx u^{\eps},\cj{u^{\eps}},\cj{\dx u^{\eps}}),\\
&u^{\eps}(0,x)=\phi(x)
\end{aligned}
\right.
\label{RNLS}
\end{equation}
for $\eps>0$, $t>0$, and $x\in\R$.
The following local existence result follows from a standard contraction mapping argument.
We omit the proof;
see, for example, Lemma 3 in \cite{Chi94}.

\begin{proposition}
\label{prop:exRNLS}
Let $\eps>0$ and $s>\frac52$.
Assume that $F(0)=0$.
Then, for any $\phi\in H^s(\R)$,
there exist $T_\eps>0$ and a unique solution
\[
u^\eps\in C([0,T_\eps];H^s(\R))
\]
to \eqref{RNLS}.
Moreover, if $\phi\in H^r(\R)$ for some $r\ge s$, then
$u^\eps\in C([0,T_\eps];H^r(\R))$ and
$u^\eps\in C((0,T_\eps];H^\infty(\R))$,
where
$H^\infty (\R) := \bigcap_{k \in \N} H^k(\R)$.
\end{proposition}

Let $\psi\in H^s(\R)$ and $\rho>0$ be fixed.
For $\phi\in B_{H^s}(\psi,\rho)$,
we derive estimates for the solution to \eqref{RNLS} on a time interval independent of $\eps\in(0,1)$,
and then pass to the limit $\eps \downarrow 0$.

Let $u^\eps$ be the solution obtained in Proposition \ref{prop:exRNLS} and set
\[
v^\eps:=\dx u^\eps,
\quad
w^\eps:=\dx^2u^\eps.
\]
For simplicity, we write
\[
F^\eps=F(u^\eps,v^\eps,\cj{u^\eps},\cj{v^\eps}),
\quad
F_\al^\eps=F_\al(u^\eps,v^\eps,\cj{u^\eps},\cj{v^\eps}),
\quad
\cdots.
\]
Then,
\begin{equation}
\dt u^\eps+(i-\eps)\dx^2u^\eps=F^\eps,
\label{ueps}
\end{equation}
\begin{equation}
\dt v^\eps+(i-\eps)\dx^2v^\eps
=F_\al^\eps v^\eps+F_\be^\eps w^\eps
+F_{\cj\al}^\eps\cj{v^\eps}+F_{\cj\be}^\eps\cj{w^\eps},
\label{veps}
\end{equation}
and
\begin{equation}
\dt w^\eps+(i-\eps)\dx^2w^\eps
=F_\be^\eps\dx w^\eps+F_{\cj\be}^\eps\cj{\dx w^\eps}+R^\eps,
\label{weps}
\end{equation}
where
\begin{equation}
\begin{aligned}
R^\eps
&:=
F_\al^\eps w^\eps+F_{\cj\al}^\eps\cj{w^\eps}
\\
&\quad+
(F_{\al\al}^\eps v^\eps+F_{\al\be}^\eps w^\eps
+F_{\al\cj\al}^\eps\cj{v^\eps}+F_{\al\cj\be}^\eps\cj{w^\eps})v^\eps
\\
&\quad+
(F_{\be \al}^\eps v^\eps+F_{\be\be}^\eps w^\eps
+F_{\be\cj\al}^\eps\cj{v^\eps}+F_{\be\cj\be}^\eps\cj{w^\eps})w^\eps
\\
&\quad+
(F_{\cj \al \al}^\eps v^\eps+F_{\cj\al \be}^\eps w^\eps
+F_{\cj\al\cj\al}^\eps\cj{v^\eps}+F_{\cj\al\cj\be}^\eps\cj{w^\eps})\cj{v^\eps}
\\
&\quad+
(F_{\cj\be \al}^\eps v^\eps+F_{\cj\be \be}^\eps w^\eps
+F_{\cj\be \cj\al}^\eps\cj{v^\eps}+F_{\cj\be\cj\be}^\eps\cj{w^\eps})\cj{w^\eps}.
\end{aligned}
\label{Reps}
\end{equation}

We eliminate the problematic part of $F_\be^\eps\dx w^\eps$ by the gauge transformation
\begin{align}
\Ld^\eps(t,x)
&:=-\frac12\int_0^x\Im F_\be^\eps(t,y)dy,
\label{Ldep}
\\
W^\eps(t,x)
&:=e^{\Ld^\eps(t,x)}w^\eps(t,x).
\label{Wga}
\end{align}

\begin{lemma}
\label{lem:gggb}
Let $s,r\in\R$ satisfy
\begin{equation}
\frac52<s\le r.
\label{condsr0}
\end{equation}
Assume that $F$ satisfies \eqref{eq:condition}.
Then,
we have
\begin{align*}
\| \Ld^\eps(t) \|_{L^\infty}
&\le C(\|u^\eps(t)\|_{H^{s-1}}),
\\
\|\dx \Ld^\eps(t)\|_{H^{r-1}}
&\le
C(\| u^\eps(t) \|_{H^{s-1}})\|u^\eps(t)\|_{H^{r}}
\end{align*}
for $\eps\in(0,1)$ and $t\in(0,T_\eps)$.
\end{lemma}

\begin{proof}
For simplicity, we suppress the time dependence.
As in \eqref{MI3},
we define
\begin{equation}
G^\eps
:=\Im F_\be^\eps
-
\Im ( F_{\be\be}(0) v^\eps + F_{\cj\be\be} (0) \cj{v^\eps} ).
\label{Gep}
\end{equation}
It follows from \eqref{MI5} that
\begin{equation}
\| G^\eps \|_{L^1}
\le C(\| u^\eps \|_{H^{s-1}}).
\label{R3g}
\end{equation}
Using \eqref{MI4}, \eqref{Ldep}, and \eqref{Gep},
we have
\begin{align*}
-2 \Ld^\eps (t,x)
&=
\Im (F_{\be \be}(0) u^\eps (t,x) + F_{\cj \be \be}(0) \cj{u^\eps} (t,x))
\\
&\quad
- \Im (F_{\be \be}(0) u^\eps (t,0) + F_{\cj \be \be}(0) \cj{u^\eps} (t,0))
\\
&\quad
+ \int_0^x G^\eps (t,y) dy.
\end{align*}
From \eqref{R3g},
we obtain
\[
\| \Ld^\eps \|_{L^\infty}
\le C(\|u^\eps\|_{H^{s-1}}).
\]

Since $\Im F_\be(0)=0$, Lemma \ref{lem:comp} with \eqref{Ldep} gives
\[
\|\dx \Ld^\eps\|_{H^{r-1}}
\les
\| \Im F_\be^\eps - \Im F_\be (0) \|_{H^{r-1}}
\le
C(\|u^\eps\|_{H^{s-1}})\|u^\eps\|_{H^{r}}.
\]
This shows the desired bounds.
\end{proof}

We define
\begin{equation}
E_r^\eps(t)
:=
\sqrt{
\|u^\eps(t)\|_{H^{r-2}}^2
+\|v^\eps(t)\|_{H^{r-2}}^2
+\|W^\eps(t)\|_{H^{r-2}}^2}
\label{Eesp}
\end{equation}
for $r \in \R$.
When $s,r$ satisfy \eqref{condsr0},
Lemmas \ref{lem:gauge} and \ref{lem:gggb} give
\begin{equation}
\begin{aligned}
E_r^\eps(t)
&\le C(\|u^\eps(t)\|_{H^s})\|u^\eps(t)\|_{H^r},
\\
\|u^\eps(t)\|_{H^r}
&\le C(E_s^\eps(t))E_r^\eps(t)
\end{aligned}
\label{Enebb}
\end{equation}
for $\eps\in(0,1)$ and $t\in(0,T_\eps)$.

Since
\[
\dx\Ld^\eps=-\frac12\Im F_\be^\eps,
\]
a direct calculation with \eqref{weps} and \eqref{Wga} gives
\begin{equation}
\dt W^\eps+(i-\eps)\dx^2W^\eps
=(\Re F_\be^\eps+\eps\Im F_\be^\eps)\dx W^\eps
+F_{\cj\be}^\eps\cj{\dx W^\eps}+\fR^\eps,
\label{Weps}
\end{equation}
where
\begin{equation}
\begin{aligned}
\fR^\eps
&:=
\big(
\dt\Ld^\eps
+ (i-\eps) (\dx^2\Ld^\eps)
+ (i-\eps) (\dx\Ld^\eps)^2 \big) W^\eps
\\
&\qquad
- (\Re F_\be^\eps + \eps \Im F_\be^\eps) (\dx\Ld^\eps)W^\eps
- F_{\cj\be}^\eps (\dx\Ld^\eps) \cj{W^\eps}
+ e^{\Ld^\eps} R^\eps.
\end{aligned}
\label{d:fR}
\end{equation}
Here,
$R^\eps$ in \eqref{d:fR}
is given by \eqref{Reps} with $w^\eps = e^{-\Ld^\eps}W^\eps$.

\begin{lemma}
\label{lem:remain}
Let $s,r\in\R$ satisfy \eqref{condsr0}.
Assume that $F$ satisfies \eqref{eq:condition}.
Then, $\fR^\eps$ defined in \eqref{d:fR} satisfies
\begin{equation}
\|\fR^\eps(t)\|_{H^{r-2}}
\le C(E_s^\eps(t))E_r^\eps(t)
\label{Remain}
\end{equation}
for $\eps\in(0,1)$ and $t\in(0,T_\eps)$.
\end{lemma}

\begin{proof}
For simplicity, we suppress the time dependence.
The only delicate term in \eqref{d:fR} is $(\dt\Ld^\eps)W^\eps$.
Indeed,
Corollary \ref{cor:comppro}, Lemmas \ref{lem:product} and \ref{lem:gggb}, and \eqref{Enebb} give
\begin{align}
\begin{split}
&\|(\dx^2\Ld^\eps)W^\eps\|_{H^{r-2}}
+
\|(\dx\Ld^\eps)^2W^\eps\|_{H^{r-2}}
\\
&\quad
+
\|F_\be^\eps(\dx\Ld^\eps)W^\eps\|_{H^{r-2}}
+
\|F_{\cj\be}^\eps\cj{(\dx\Ld^\eps)W^\eps}\|_{H^{r-2}}
\\
&
\le C(E_s^\eps)E_r^\eps.
\end{split}
\label{fR1}
\end{align}
Moreover,
Corollary \ref{cor:comppro},
Lemmas \ref{lem:gauge} and \ref{lem:gggb},
and \eqref{Enebb} yield that
\begin{equation}
\|e^{\Ld^\eps}R^\eps\|_{H^{r-2}}
\le C(E_s^\eps)E_r^\eps.
\label{fR6}
\end{equation}

We estimate $\dt\Ld^\eps$.
Using \eqref{ueps} and \eqref{veps},
we write
\begin{equation}
\dt u^\eps=R_1^\eps,
\quad
\dt v^\eps=(-i+\eps)\dx^2v^\eps+R_2^\eps,
\label{eq:uve}
\end{equation}
where
\begin{align}
\begin{split}
R_1^\eps&:=(-i+\eps)\dx v^\eps+F^\eps,
\\
R_2^\eps&:=
F_\al^\eps v^\eps+F_\be^\eps\dx v^\eps
+F_{\cj\al}^\eps\cj{v^\eps}+F_{\cj\be}^\eps\cj{\dx v^\eps}.
\end{split}
\label{R2}
\end{align}
Since $F(0)=0$,
Lemma \ref{lem:comp} with \eqref{Enebb} yields that
\begin{equation}
\| R_1^\eps \|_{H^{r-2}} + \| R_2^\eps \|_{H^{r-2}}
\le
C(E_s^\eps) E_r^\eps.
\label{N2L1a}
\end{equation}

Set
\begin{align}
Q^\eps
&:=
F_\al(0)u^\eps+F_\be(0)v^\eps
+F_{\cj\al}(0)\cj{u^\eps}+F_{\cj\be}(0)\cj{v^\eps},
\label{R2q}
\\
\begin{split}
\wt R_2^\eps
&:= R_2^\eps -\dx Q^\eps
=
(F_\al^\eps - F_\al(0))v^\eps
+ ( F_\be^\eps -F_\be(0)) \dx v^\eps
\\
&\hspace*{90pt}
+ (F_{\cj\al}^\eps - F_{\cj\al}(0)) \cj{v^\eps}
+ (F_{\cj\be}^\eps - F_{\cj\be}(0)) \cj{\dx v^\eps}.
\end{split}
\label{R2q2}
\end{align}
It follows from \eqref{Enebb} that
\begin{equation}
\|Q^\eps\|_{L^\infty}
+\|\wt R_2^\eps\|_{L^1}
\le C(E_s^\eps).
\label{N2L1}
\end{equation}
A direct calculation with \eqref{Gep} and \eqref{eq:uve}--\eqref{R2q2} gives
\begin{equation}
\begin{aligned}
\dt\Im F_\be^\eps
&=
\Im\big(
(-i+\eps)F_{\be\be}(0)\dx^2v^\eps
+(i+\eps)F_{\cj\be\be}(0)\cj{\dx^2v^\eps}
\big)
\\
&\quad
+ \dx \Im\big(F_{\be\be}(0) Q^\eps + F_{\cj\be\be}(0) \cj{Q^\eps}\big)
\\
&\quad
+ (-i+\eps)\dx(G_\be^\eps\dx v^\eps)
+ (i+\eps)\dx(G_{\cj\be}^\eps\cj{\dx v^\eps})
+ R_3^\eps,
\end{aligned}
\label{Tabbep}
\end{equation}
where
\begin{equation}
\begin{split}
R_3^\eps
&:=
\Im\big(F_{\be\be}(0) \wt R_2^\eps+F_{\cj\be\be}(0)\cj{\wt R_2^\eps}\big)
\\
&\quad
+
G_{\al}^\eps R_1^\eps + G_{\cj \al}^\eps \cj{R_1^\eps}
+
G_{\be}^\eps R_2^\eps + G_{\cj \be}^\eps \cj{R_2^\eps}
\\
&\quad
- (-i+\eps) (\dx G_\be^\eps) \dx v^\eps
- (i+\eps) (\dx G_{\cj\be}^\eps) \cj{\dx v^\eps}.
\end{split}
\label{R3}
\end{equation}
Combining \eqref{N2L1}, \eqref{Gep}, and \eqref{R3g},
we obtain
\begin{equation}
\| R_3^\eps \|_{L^1}
\le C(E_s^\eps).
\label{N2L5}
\end{equation}

Integrating \eqref{Tabbep} from $0$ to $x$ and using \eqref{Ldep}, we obtain
\begin{equation}
-2\dt\Ld^\eps=\Nl_1^\eps+\Nl_2^\eps,
\label{intGep}
\end{equation}
where
\begin{align*}
\Nl_1^\eps
&:=
\Im\big(
(-i+\eps)F_{\be\be}(0)\dx v^\eps
+(i+\eps)F_{\cj\be\be}(0)\cj{\dx v^\eps}
\big)
\\
&\quad
+ \Im\big(F_{\be\be}(0) Q^\eps + F_{\cj\be\be}(0) \cj{Q^\eps}\big)
\\
&\quad
+(-i+\eps)G_\be^\eps\dx v^\eps
+(i+\eps)G_{\cj\be}^\eps\cj{\dx v^\eps},
\\
\Nl_2^\eps(t,x)
&:=-\Nl_1^\eps(t,0)+\int_0^xR_3^\eps(t,y)dy.
\end{align*}
Corollary \ref{cor:comppro}, \eqref{R2q}, and  \eqref{Enebb} give
\[
\| \Nl_1^\eps \|_{H^{r-2}}
\le C(E_s^\eps) E_r^\eps.
\]
From Lemma \ref{lem:product},
we have
\begin{equation}
\begin{aligned}
\| \Nl_1^\eps W^\eps\|_{H^{r-2}}
&\les
\| \Nl_1^\eps \|_{H^{s-2}} \| W^\eps\|_{H^{r-2}}
+ \| \Nl_1^\eps \|_{H^{r-2}} \| W^\eps\|_{H^{s-2}}
\\
&\le
C(E_s^\eps) E_r^\eps.
\end{aligned}
\label{Nl1ep5}
\end{equation}

Corollary \ref{cor:comppro}, \eqref{R2q}, \eqref{Enebb}, and \eqref{N2L5} give
\[
\| \Nl_2^\eps \|_{L^\infty}
\le C(E_s^\eps).
\]
Moreover,
it follows from
$\dx\Nl_2^\eps=R_3^\eps$,
Lemma \ref{lem:comp}, \eqref{R3}, \eqref{Enebb}, and \eqref{N2L1a} that
\[
\| \dx^{[r]-1} \Nl_2^\eps \|_{L^2}
=
\| \dx^{[r]-2} R_3^\eps \|_{L^2}
\le
\| R_3^\eps \|_{H^{r-2}}
\le C(E_s^\eps)E_r^\eps.
\]
Hence,
Lemma \ref{lem:prod2} implies that
\begin{equation}
\begin{aligned}
\| \Nl_2^\eps W^\eps\|_{H^{r-2}}
&\les
\| W^\eps \|_{H^{r-2}} \| \Nl_2^\eps \|_{L^\infty}
+
\| W^\eps \|_{H^{s-2}} \| \dx^{[r]-1} \Nl_2^\eps \|_{L^2}
\\
&\le
C(E_s^\eps) E_r^\eps.
\end{aligned}
\label{Nl1ep6}
\end{equation}

From \eqref{intGep}--\eqref{Nl1ep6},
we have
\[
\| (\dt \Ld^\eps) W^\eps\|_{H^{r-2}}
\le
C(E_s^\eps) E_r^\eps.
\]
Together with \eqref{fR1} and \eqref{fR6}, this proves \eqref{Remain}.
\end{proof}

We now record the following estimates for the first-order terms.

\begin{lemma}
\label{lem:ened}
Let $s,r\in\R$ satisfy \eqref{condsr0}.

\begin{enumerate}
\item
Assume that $A : \R\to\R$ satisfies
\[
A-A_0 \in H^{r-1}(\R)
\]
for some $A_0 \in \R$.
Then,
we have
\begin{align*}
&| \Re (A\dx f,f)_{H^{r-2}} |
\\
&\les
\|A-A_0\|_{H^{s-1}} \|f\|_{H^{r-2}}^2
+
\|A-A_0\|_{H^{r-1}}
\|f\|_{H^{s-2}}\|f\|_{H^{r-2}}
\end{align*}
for $f\in H^{r-1}(\R)$.

\item
Assume that $B : \R\to\C$ satisfies
\[
B-B_0\in H^{r-1}(\R)
\]
for some $B_0\in\C$.
Then,
we have
\begin{align*}
&| (B\cj{\dx f},f)_{H^{r-2}} |
\\
&\les
\|B-B_0\|_{H^{s-1}}
\|f\|_{H^{r-2}}^2
+
\|B-B_0\|_{H^{r-1}}
\|f\|_{H^{s-2}}\|f\|_{H^{r-2}}
\end{align*}
for $f\in H^{r-1}(\R)$.
\end{enumerate}
\end{lemma}

\begin{proof}
If $g \in W^{1,\infty}(\R)$ is real-valued,
integration by parts yields that
\[
\Re(g J^{r-2}\dx f, J^{r-2} f)_{L^2}
= - \frac 12 \int_\R (\dx g) |J^{r-2} f|^2 dx.
\]
We have
\begin{align*}
&\Re (A\dx f,f)_{H^{r-2}}
\\
&=
\Re ( (A-A_0) \dx f,f)_{H^{r-2}}
\\
&=
\Re ( [J^{r-2}, A-A_0] \dx f, J^{r-2}f)_{L^2}
- \frac 12 \int_\R (\dx (A-A_0)) |J^{r-2} f|^2 dx.
\end{align*}
It follows from Lemma \ref{lem:Com1} and $s-1>\frac 32$ that
\begin{align*}
\|[J^{r-2},A-A_0]\dx f\|_{L^2}
&\les
\| A-A_0 \|_{H^{s-1}} \| f \|_{H^{r-2}}
+ \| A-A_0 \|_{H^{r-1}} \| f \|_{H^{s-2}}
.
\end{align*}
Hence, we obtain
\[
|\Re (A\dx f,f)_{H^{r-2}}|
\les
\|A-A_0\|_{H^{s-1}} \| f\|_{H^{r-2}}^2
+
\| A-A_0 \|_{H^{r-1}} \| f \|_{H^{s-2}} \| f \|_{H^{r-2}}.
\]
This proves (i).

If $h \in W^{1,\infty}(\R)$,
integration by parts yields that
\[
(h J^{r-2} \cj{\dx f}, J^{r-2} f)_{L^2}
= - \frac 12 \int_\R (\dx h) (J^{r-2} \cj f)^2 dx.
\]
Hence, the same calculation as above yields that
\[
|(B \cj{\dx f},f)_{H^{r-2}}|
\les
\| B-B_0 \|_{H^{s-1}} \| f\|_{H^{r-2}}^2
+
\| B-B_0 \|_{H^{r-1}} \| f \|_{H^{s-2}} \| f \|_{H^{r-2}}.
\]
This proves (ii).
\end{proof}

\begin{proposition}
\label{prop:dtEne}
Let $s,r\in\R$ satisfy \eqref{condsr0}.
Assume that $F$ satisfies \eqref{eq:condition}.
Then,
there exists a non-decreasing function
$C_{0,r}:[0,\infty)\to(0,\infty)$ such that
\[
\dt E_r^\eps(t)
\le C_{0,r}(E_s^\eps(t))E_r^\eps(t)
\]
for $\eps\in(0,1)$ and $t\in(0,T_\eps)$.
\end{proposition}

\begin{proof}
From \eqref{eq:uve} and \eqref{N2L1a},
we have
\begin{equation}
\dt\big(
\|u^\eps\|_{H^{r-2}}^2+\|v^\eps\|_{H^{r-2}}^2
\big)
\le
C(E_s^\eps)E_r^\eps
\big(
\|u^\eps\|_{H^{r-2}}+\|v^\eps\|_{H^{r-2}}
\big).
\label{est:u1}
\end{equation}
For $W^\eps$,
\eqref{Weps} gives
\begin{align*}
\frac 12 \dt (\|W^\eps\|_{H^{r-2}}^2)
&=
-\eps \| \dx W^\eps \|_{H^{r-2}}^2
+ \Re ((\Re F_\be^\eps+\eps\Im F_\be^\eps)\dx W^\eps, W^\eps)_{H^{r-2}}
\\
&\quad
+ \Re (F_{\cj\be}^\eps\cj{\dx W^\eps}, W^\eps)_{H^{r-2}}
+ \Re (\fR^\eps, W^\eps)_{H^{r-2}}.
\end{align*}
For the second and third terms,
we apply Lemma \ref{lem:ened} with
\[
A=\Re F_\be^\eps+\eps\Im F_\be^\eps,
\quad
B=F_{\cj\be}^\eps,
\quad
A_0=F_\be(0),
\quad
B_0=F_{\cj\be}(0).
\]
For the fourth term,
we use Lemma \ref{lem:remain}.
Thus,
we obtain
\begin{equation}
\frac 12 \dt (\|W^\eps\|_{H^{r-2}}^2)
\le
-\eps\|\dx W^\eps\|_{H^{r-2}}^2
+ C(E_s^\eps) E_r^\eps \|W^\eps\|_{H^{r-2}}.
\label{enWe}
\end{equation}
Combining \eqref{Eesp}, \eqref{est:u1}, and \eqref{enWe} proves the assertion.
\end{proof}

We next establish a uniform lower bound on the lifespan and a uniform persistence estimate.

\begin{corollary}
\label{cor:enebdur}
Let $s,r\in\R$ satisfy \eqref{condsr0} and let $M>0$.
Assume that $F$ satisfies \eqref{eq:condition}.
Then, there exist $T=T(M)>0$ and $C_{1,r}(M)>0$ such that,
for every $\phi\in H^r(\R)$ satisfying $\|\phi\|_{H^s}\le M$,
the solution $u^\eps$ to \eqref{RNLS} exists on $[0,T]$ and satisfies
\[
\|u^\eps\|_{L_T^\infty H^r}
\le C_{1,r}(M)\|\phi\|_{H^r}
\]
for $\eps\in(0,1)$.
\end{corollary}

\begin{proof}
Let $T_\eps^{\max}$ be the maximal existence time and set
\[
K:=E_s^\eps(0).
\]
Since $u^\eps(0)=\phi$, $K$ is independent of $\eps$, and \eqref{Enebb} gives
\[
K\le C(M)=:K'.
\]
Set
\[
T=T(M):=\frac{\log2}{C_{0,s}(2K')}.
\]
Proposition \ref{prop:dtEne} and the usual bootstrap argument yield
\[
E_s^\eps(t)\le 2K
\]
for $0\le t<\min(T,T_\eps^{\max})$.
If $T_\eps^{\max}\le T$, \eqref{Enebb} gives a uniform $H^s$ bound up to $T_\eps^{\max}$,
contradicting the continuation criterion for \eqref{RNLS}.
Hence, $T_\eps^{\max}>T$.
Applying Proposition \ref{prop:dtEne} at regularity $r$ and using \eqref{Enebb}, we obtain
\[
\|u^\eps\|_{L_T^\infty H^r}
\le C_{1,r}(M)\|\phi\|_{H^r}.
\]
This concludes the proof.
\end{proof}

By Corollary \ref{cor:enebdur} and the Banach-Alaoglu theorem,
there exists a sequence $\{ \eps_j \}_{j \in \N} \subset (0,1)$ such that
$\eps_j \downarrow 0$ and
$\{ u^{\eps_j} \}_{j \in \N}$ converges
to some $u$
in the weak-$\ast$ topology of $L^\infty([0,T];H^s(\R))$.
Moreover, \eqref{ueps} and Lemma \ref{lem:comp} give
\[
\sup_{j \in \N} \|\dt u^{\eps_j}\|_{L_T^\infty H^{s-2}}<\infty.
\]
By Lemma \ref{lem:cpt}, after taking a subsequence,
we have
\[
u^{\eps_j}\to u,
\quad
\dx u^{\eps_j}\to\dx u
\]
locally uniformly on $[0,T]\times\R$.
Consequently,
\[
F(u^{\eps_j},\dx u^{\eps_j},
\cj{u^{\eps_j}},\cj{\dx u^{\eps_j}})
\to
F(u,\dx u,\cj u,\cj{\dx u})
\]
locally uniformly.
Moreover,
\[
\eps_j\dx^2u^{\eps_j}\to0
\]
in the sense of distributions.
Since
\[
u^{\eps_j}(0)=\phi,
\]
we also have $u(0)=\phi$.
Passing to the limit in \eqref{RNLS}, we conclude that
$u$ is a solution to \eqref{NLS} in the sense of
Definition \ref{def:sol}.

If $\phi\in H^r(\R)$ for some $r\ge s$, Corollary \ref{cor:enebdur} also gives
\[
\sup_{j \in \N}
\|u^{\eps_j}\|_{L_T^\infty H^r}
\le
C_{1,r}(M)\|\phi\|_{H^r}.
\]
Passing to a further subsequence,
we may assume that
$\{ u^{\eps_j} \}_{j \in \N}$ converges
in the weak-$\ast$ topology of $L^\infty([0,T];H^r(\R))$.
The weak-$\ast$ limit agrees with $u$ by uniqueness of distributional limits.
Hence, by weak lower semicontinuity,
we have
\begin{equation}
\|u\|_{L_T^\infty H^r}
\le
C_{1,r}(M)\|\phi\|_{H^r}.
\label{persbound}
\end{equation}

We next prove uniqueness and the lower-order difference estimate.

\begin{proposition}
\label{prop:diffHs-1s}
Let $s>\frac52$ and assume that $F$ satisfies \eqref{eq:condition}.
Let $u^{(j)}\in C_w ([0,T];H^s(\R))$ be solutions to \eqref{NLS}
with initial data $\phi^{(j)}\in H^s(\R)$ for $j=1,2$.
Set
\[
M:= \|u^{(1)}\|_{L_T^\infty H^s}+\|u^{(2)}\|_{L_T^\infty H^s}.
\]
Then,
\[
\|u^{(1)}-u^{(2)}\|_{L_T^\infty H^{1}}
\le
C_1 (T,M)
\|\phi^{(1)}-\phi^{(2)}\|_{H^{1}}.
\]
In particular, the solution to \eqref{NLS} is unique.
\end{proposition}

\begin{proof}
Set
\[
v^{(j)}:=\dx u^{(j)},
\quad
F^{(j)}:=F(u^{(j)},v^{(j)},\cj{u^{(j)}},\cj{v^{(j)}})
\]
for $j=1,2$.
We write
\[
\breve u:=u^{(1)}-u^{(2)},
\quad
\breve v:=\dx\breve u.
\]
Since $u^{(1)}$ and $u^{(2)}$ are solutions to \eqref{NLS},
we have
\begin{equation}
\dt \breve u+i\dx^2\breve u
= F^{(1)}-F^{(2)},
\label{NLSdu}
\end{equation}
and
\begin{equation}
\begin{aligned}
\dt\breve v+i\dx^2\breve v
&=F_\be^{(1)}\dx\breve v+F_{\cj\be}^{(1)}\cj{\dx\breve v}
\\
&\quad+(F_\be^{(1)}-F_\be^{(2)})\dx v^{(2)}
+(F_{\cj\be}^{(1)}-F_{\cj\be}^{(2)})\cj{\dx v^{(2)}}+\breve R,
\end{aligned}
\label{NLSdv}
\end{equation}
where
\[
\breve R
:=F_\al^{(1)}\breve v+F_{\cj\al}^{(1)}\cj{\breve v}
+(F_\al^{(1)}-F_\al^{(2)})v^{(2)}
+(F_{\cj\al}^{(1)}-F_{\cj\al}^{(2)})\cj{v^{(2)}}.
\]

Define
\[
\Ld^{(1)}(t,x):=-\frac12\int_0^x\Im F_\be^{(1)}(t,y)dy,
\quad
\breve V:=e^{\Ld^{(1)}}\breve v.
\]
With \eqref{NLSdv},
we have
\begin{equation}
\begin{aligned}
\dt\breve V+i\dx^2\breve V
&=(\Re F_\be^{(1)})\dx\breve V
+F_{\cj\be}^{(1)}\cj{\dx\breve V}
+\breve H+\breve\fR,
\end{aligned}
\label{NLSdW0}
\end{equation}
where
\[
\breve H
:=e^{\Ld^{(1)}}(F_\be^{(1)}-F_\be^{(2)})\dx v^{(2)}
+e^{\Ld^{(1)}}(F_{\cj\be}^{(1)}-F_{\cj\be}^{(2)})\cj{\dx v^{(2)}},
\]
and
\begin{equation}
\begin{aligned}
\breve\fR
&:=
\big(\dt\Ld^{(1)}+i(\dx\Ld^{(1)})^2+i\dx^2\Ld^{(1)}\big)\breve V
\\
&\quad
-(\Re F_\be^{(1)})(\dx\Ld^{(1)})\breve V
-F_{\cj\be}^{(1)}\cj{(\dx\Ld^{(1)})\breve V}
+e^{\Ld^{(1)}} \breve R.
\end{aligned}
\label{wtfR00}
\end{equation}
Set
\begin{equation}
\breve E_{1}(t)
:=
\sqrt{
\|\breve u(t)\|_{L^2}^2
+\|\breve V(t)\|_{L^2}^2}.
\label{Ewt0}
\end{equation}

The argument in the proof of Lemma \ref{lem:gggb} gives
\begin{align*}
\| \Ld^{(1)}(t) \|_{L^\infty}
&\le
C(\| u^{(1)}(t) \|_{H^{s-1}}).
\end{align*}
Consequently,
we have
\begin{equation}
\begin{split}
&\breve E_{1}(t)
\le
C(\|u^{(1)} (t)\|_{H^s})
\|\breve u(t)\|_{H^{1}},
\\
&
\|\breve u(t)\|_{L^2}
+ \|\breve v(t)\|_{L^2}
\le
C(\|u^{(1)} (t)\|_{H^s})\breve E_{1}(t).
\end{split}
\label{diffbd1}
\end{equation}
Since $s>\frac 52$,
the Sobolev embedding yields
\[
\sum_{j=1,2}
(\| u^{(j)} (t) \|_{L^\infty}
+ \| \dx v^{(j)} (t) \|_{L^\infty})
\le C(M).
\]
Since $F(0)=0$,
Remark \ref{rem:sold} gives
\[
u^{(1)} \in
C_w ([0,T];H^s(\R)) \cap AC([0,T]; H^{s-2}(\R)).
\]
The derivation of \eqref{intGep}
remains valid in the distributional sense
when we set $\eps=0$ and replace $u^\eps$ by $u^{(1)}$.
Hence, the estimates obtained in the proof of Lemma \ref{lem:remain}
give
\[
\|\dt\Ld^{(1)}(t)\|_{L^\infty}\le C(M).
\]
Hence,
we have
\begin{equation}
\|F^{(1)}(t)-F^{(2)}(t)\|_{L^2}
+
\|\breve H(t) \|_{L^2}
+
\|\breve\fR (t) \|_{L^2}
\le C(M)\breve E_{1}(t).
\label{diffrem}
\end{equation}

Since we work in $L^2(\R)$,
the commutator terms in the proof of Lemma \ref{lem:ened} vanish.
Hence,
integrating by parts as in the proof of Lemma \ref{lem:ened}
and using
\eqref{Ewt0}, \eqref{NLSdu}, \eqref{NLSdW0}, and \eqref{diffrem},
we obtain
\[
\dt \breve E_{1}(t)
\le C(M)\breve E_{1}(t)
\]
for almost every $t\in(0,T)$.
Moreover, \eqref{diffbd1} at $t=0$ gives
\[
\breve E_{1}(0)
\le C(\|\phi^{(1)}\|_{H^s})
\|\phi^{(1)}-\phi^{(2)}\|_{H^{1}}.
\]
Gronwall's inequality and \eqref{diffbd1} prove the assertion.
\end{proof}

\subsection{Continuous dependence on initial data}
\label{Subsec:code}

We use a Bona-Smith-type approximation as in \cite{BoSm75}
to prove continuous dependence.
First, we prove the following energy estimate.

\begin{proposition}
\label{prop:diffNM}
Let
\begin{equation}
s>\frac52,
\quad
s_0 \in \Big( \frac 52, s \Big),
\quad
\ta := \frac{s-s_0+1}{s-1} \in (0,1).
\label{BS001}
\end{equation}
Assume that $F$ satisfies \eqref{eq:condition}.
Let $u^{(j)}\in C_w ([0,T];H^{s+2}(\R))$ be solutions to \eqref{NLS}
with initial data $\phi^{(j)}\in H^{s+2}(\R)$ for $j=1,2$.
Set
\[
M:=\|u^{(1)}\|_{L_T^\infty H^s}+\|u^{(2)}\|_{L_T^\infty H^s}.
\]
Then,
we have
\begin{equation}
\begin{split}
\|u^{(1)}-u^{(2)}\|_{L_T^\infty H^s}
&\le
C(T,M)
\big(
\|\phi^{(1)}-\phi^{(2)}\|_{H^s}
\\
&\qquad
+
\|u^{(2)}\|_{L_T^\infty H^{s+1}}
\|\phi^{(1)}-\phi^{(2)}\|_{H^1}^\ta
\big).
\end{split}
\label{diffNM}
\end{equation}
\end{proposition}

\begin{proof}
For $j=1,2$, set
\[
v^{(j)}:=\dx u^{(j)},
\quad
w^{(j)}:=\dx^2 u^{(j)},
\quad
F^{(j)} := F(u^{(j)},v^{(j)},\cj{u^{(j)}},\cj{v^{(j)}}).
\]
We write
\[
\wt u:=u^{(1)}-u^{(2)},
\quad
\wt v:=\dx\wt u,
\quad
\wt w:=\dx^2\wt u,
\]
and define
\[
\Ld^{(1)}(t,x):=-\frac12\int_0^x\Im F_\be^{(1)}(t,y)dy,
\quad
\wt W:=e^{\Ld^{(1)}}\wt w.
\]
As in \eqref{weps},
we write
\[
\dt w^{(j)}+ i\dx^2w^{(j)}
= F_\be^{(j)} \dx w^{(j)} + F_{\cj\be}^{(j)} \cj{\dx w^{(j)}}+R^{(j)},
\]
where $R^{(j)}$ is given by the same expression as \eqref{Reps}
with $u^\eps$ replaced by $u^{(j)}$.
Then,
\[
\dt\wt W+i\dx^2\wt W
=(\Re F_\be^{(1)})\dx\wt W
+F_{\cj\be}^{(1)}\cj{\dx\wt W}
+\wt H+\wt\fR,
\]
where
\[
\wt H
:=
e^{\Ld^{(1)}}(F_\be^{(1)}-F_\be^{(2)})\dx w^{(2)}
+ e^{\Ld^{(1)}}(F_{\cj\be}^{(1)}-F_{\cj\be}^{(2)}) \cj{\dx w^{(2)}},
\]
and $\wt\fR$ is given by the same expression as \eqref{wtfR00},
with $\breve V$ and $\breve R$ replaced by $\wt W$ and $R^{(1)}-R^{(2)}$, respectively.
Set
\begin{equation}
\wt E_s(t)
:=
\sqrt{
\|\wt u(t)\|_{H^{s-2}}^2
+\|\wt v(t)\|_{H^{s-2}}^2
+\|\wt W(t)\|_{H^{s-2}}^2}.
\label{Ewt}
\end{equation}
Arguing as in the derivation of \eqref{diffbd1} (see also \eqref{Enebb}),
we have
\begin{equation}
\wt E_s(t)
\le C(M) \|\wt u(t)\|_{H^s},
\quad
\|\wt u(t)\|_{H^s}
\le C(M)\wt E_s(t).
\label{diffC0}
\end{equation}

Repeating the argument in Lemma \ref{lem:remain} with $\eps=0$ and $r=s$,
and using Lemmas \ref{lem:product} and \ref{lem:comp} together with \eqref{diffC0},
we obtain
\[
\| \wt \fR (t) \|_{H^{s-2}}
\le C(M) \wt E_s(t).
\]
The only term involving a loss of derivatives is $\wt H$.
For the first term in $\wt H$, we write
\[
\begin{aligned}
&\Re \big( e^{\Ld^{(1)}}(F_\be^{(1)}-F_\be^{(2)})\dx w^{(2)},\wt W \big)_{H^{s-2}}
\\
&=
\Re \big( [J^{s-2},e^{\Ld^{(1)}}(F_\be^{(1)}-F_\be^{(2)})]\dx w^{(2)},J^{s-2}\wt W \big)_{L^2}
\\
&\quad+
\Re \big( e^{\Ld^{(1)}}(F_\be^{(1)}-F_\be^{(2)})J^{s-2}\dx w^{(2)},J^{s-2}\wt W \big)_{L^2}
\\
&=: H_1+H_2
.
\end{aligned}
\]
By Lemmas \ref{lem:Com1} and \ref{lem:comp},
together with \eqref{diffC0},
we obtain
\begin{align*}
|H_1(t)|
&\les
\big(
\| e^{\Ld^{(1)}(t)}(F_\be^{(1)}(t)-F_\be^{(2)}(t)) \|_{H^{s_0-1}} \| w^{(2)}(t) \|_{H^{s-2}}
\\
&\qquad
+
\| e^{\Ld^{(1)}(t)}(F_\be^{(1)}(t)-F_\be^{(2)}(t)) \|_{H^{s-1}} \| w^{(2)}(t) \|_{H^{s_0-2}}
\big) \| \wt W(t) \|_{H^{s-2}}
\\
&\le
C(M)\wt E_s(t) \|\wt W(t) \|_{H^{s-2}}.
\end{align*}
By Lemma \ref{lem:comp} and \eqref{diffC0},
we have
\begin{align*}
|H_2(t)|
&\les
\| e^{\Ld^{(1)}(t)} \|_{L^\infty}
\| F_\be^{(1)}(t)-F_\be^{(2)}(t) \|_{H^{s_0-2}} \| w^{(2)}(t) \|_{H^{s-1}} \| \wt W(t) \|_{H^{s-2}}
\\
&\le
C(T,M)
\| u^{(2)}(t) \|_{H^{s+1}}
\| \wt u(t)\|_{H^{s_0-1}}
\|\wt W(t) \|_{H^{s-2}}.
\end{align*}
The conjugate term is treated similarly.
Using \eqref{BS001} and Proposition \ref{prop:diffHs-1s},
we obtain by interpolation
\[
\|\wt u\|_{L_T^\infty H^{s_0-1}}
\les
\|\wt u\|_{L_T^\infty H^1}^{\ta}
\|\wt u\|_{L_T^\infty H^{s}}^{1-\ta}
\le
C(T,M)
\|\phi^{(1)}-\phi^{(2)}\|_{H^1}^{\ta}
.
\]

Thus, Lemma \ref{lem:ened} with \eqref{Ewt} gives
\[
\dt\wt E_s(t)
\le
C(M)\wt E_s(t)
+C(T,M)
\| u^{(2)} \|_{L_T^\infty H^{s+1}}
\| \phi^{(1)}-\phi^{(2)}\|_{H^1}^\ta.
\]
Since \eqref{diffC0} at $t=0$ gives
\[
\wt E_s(0)
\le C(M)\|\phi^{(1)}-\phi^{(2)}\|_{H^s},
\]
Gronwall's inequality proves \eqref{diffNM}.
\end{proof}

Let $\chi\in C_c^\infty(\R)$ satisfy
$0\le\chi\le1$, $\chi(\xi)=1$ on $|\xi|\le1$, and $\supp\chi\subset[-2,2]$.
For $\eta\in(0,1]$, set
\[
\ft{\Pi_\eta f}(\xi):=\chi(\eta\xi)\ft f(\xi).
\]
We record some elementary properties of $\Pi_\eta$.

\begin{lemma}
\label{lem:BoSm}
Let $s \in \R$ and $\phi \in H^s(\R)$.
Then, $\Pi_\eta$ satisfies the following:
\begin{enumerate}
\item
$\Pi_\eta \phi \in H^\infty(\R)$ for $\eta\in(0,1]$.

\item
$\displaystyle\lim_{\eta \downarrow 0}\|\Pi_\eta\phi-\phi\|_{H^s}=0$.

\item
If $r>s$, then $\|\Pi_\eta\phi\|_{H^r}\les\eta^{s-r}\|\phi\|_{H^s}$.

\item
If $r<s$, then $\|\Pi_\eta\phi-\phi\|_{H^r}\les\eta^{s-r}\|\phi\|_{H^s}$.
\end{enumerate}
\end{lemma}

Finally, we prove local well-posedness.

\begin{theorem}
\label{thm:suff}
Let $s>\frac52$ and assume that $F$ satisfies \eqref{eq:condition}.
Then, the Cauchy problem \eqref{NLS} is locally well-posed in $H^s(\R)$.
\end{theorem}

\begin{proof}
Fix $\psi\in H^s(\R)$ and $\rho>0$.
Set
\[
M := \| \psi \|_{H^s} + \rho.
\]
By Corollary \ref{cor:enebdur} and the compactness argument above,
there exists $T = T(M)>0$ such that,
for every initial datum $\phi \in H^s(\R)$ with $\| \phi \|_{H^s} \le M$,
the Cauchy problem \eqref{NLS} has a solution on $[0,T]$ satisfying \eqref{persbound} with $r=s$.
Since $B_{H^s}(\psi,\rho) \subset B_{H^s}(0,M)$,
the required existence on $[0,T]$ follows.
Uniqueness follows from Proposition \ref{prop:diffHs-1s}.

Note that
\[
\| \Pi_\eta \phi \|_{H^s}
\le
\| \phi \|_{H^s}
\le M
\]
for $\phi \in B_{H^s}(\psi,\rho)$ and $0<\eta \le 1$.
Thus,
the same time interval $[0,T]$ can be used for the original and smoothed initial data constructed below.

Let $\phi^{(\l)}\to\phi$ in $H^s(\R)$ with
$\phi^{(\l)},\phi\in B_{H^s}(\psi,\rho)$,
and denote the corresponding solutions by $u^{(\l)}$ and $u$.
For fixed $\eta \in (0,1]$, set
\[
\phi_\eta^{(\l)}:=\Pi_\eta\phi^{(\l)},
\quad
\phi_\eta:=\Pi_\eta\phi,
\]
and let $u_\eta^{(\l)}$, $u_\eta$ be the corresponding solutions.
By Proposition \ref{prop:diffNM},
we have
\begin{equation}
\lim_{\l\to\infty}u_\eta^{(\l)}=u_\eta
\quad\text{in }L^\infty([0,T];H^s(\R)).
\label{convNN}
\end{equation}

It follows from \eqref{persbound} and Lemma \ref{lem:BoSm} that
\[
\|u_\eta\|_{L_T^\infty H^{s+1}}
\le C(M)\eta^{-1}\|\phi\|_{H^s}.
\]
Let $s_0 \in (\frac 52,s)$.
By Proposition \ref{prop:diffNM} and Lemma \ref{lem:BoSm},
we have
\[
\| u_{\eta'}-u_\eta \|_{L_T^\infty H^s}
\le
C(M)
\big(
\|\phi_\eta-\phi_{\eta'}\|_{H^s}
+ \eta^{s-s_0}
\big)
\]
for $0<\eta'<\eta\le 1$.
Hence,
the limit $\wt u := \lim_{\eta \downarrow 0} u_\eta$ exists in $L^\infty([0,T];H^s(\R))$.
Passing to the limit in the equation shows that $\wt u$ is an $H^s(\R)$-solution to \eqref{NLS} with initial datum $\phi$.
By Proposition \ref{prop:diffHs-1s}, $\wt u$ agrees with the solution $u$ constructed above.
Therefore,
\begin{equation}
\|u-u_\eta\|_{L_T^\infty H^s}
\le
C(M)
\big(
\|\phi-\phi_\eta\|_{H^s}
+\eta^{s-s_0}
\big).
\label{BSdiffb}
\end{equation}
Similarly,
we have
\begin{equation}
\|u^{(\l)}-u_\eta^{(\l)}\|_{L_T^\infty H^s}
\le
C(M)
\big(
\|\phi^{(\l)}-\phi_\eta^{(\l)}\|_{H^s}
+\eta^{s-s_0}
\big)
\label{BSdiffc}
\end{equation}
for sufficiently large $\l$.

Since
$\phi_\eta - \phi_\eta^{(\l)}
= \Pi_\eta (\phi-\phi^{(\l)})$,
we have
\begin{equation}
\begin{aligned}
\|\phi^{(\l)}-\phi_\eta^{(\l)}\|_{H^s}
&\le
\|\phi-\phi^{(\l)}\|_{H^s}
+ \| \phi - \phi_\eta^{(\l)}\|_{H^s}
\\
&\le
2 \|\phi-\phi^{(\l)}\|_{H^s}
+ \| \phi - \phi_\eta\|_{H^s}.
\end{aligned}
\label{BSdiffd}
\end{equation}
It follows from \eqref{BSdiffb}--\eqref{BSdiffd} that
\begin{align*}
\|u^{(\l)}-u\|_{L_T^\infty H^s}
&\le
\|u^{(\l)}-u_\eta^{(\l)}\|_{L_T^\infty H^s}
+\|u_\eta^{(\l)}-u_\eta\|_{L_T^\infty H^s}
+\|u_\eta-u\|_{L_T^\infty H^s}
\\
&\le
\|u_\eta^{(\l)}-u_\eta\|_{L_T^\infty H^s}
\\
&\quad
+
C(M)
\big(
\|\phi-\phi^{(\l)}\|_{H^s}
+ \| \phi - \phi_\eta\|_{H^s}+\eta^{s-s_0}
\big)
.
\end{align*}
First choose $\eta>0$ sufficiently small and then let $\l\to\infty$.
With \eqref{convNN},
this proves
\[
\lim_{\l\to\infty}u^{(\l)}=u
\quad\text{in }L^\infty([0,T];H^s(\R)).
\]
This concludes the proof.
\end{proof}

\section{Non-existence}
\label{sec:necessity}

In this section,
we prove the non-existence result when at least one of the conditions in \eqref{eq:condition} fails.
At the end of this section,
we complete the proof of Theorem \ref{thm:main}.

\subsection{The constant-term obstruction}

\begin{proposition}
\label{prop:F0}

Let $s>\frac 32$ and assume
\begin{align*}
F(0)\ne0.
\end{align*}
Then, for every
$\phi\in H^s(\R)$,
the Cauchy problem \eqref{NLS} has no solution.
\end{proposition}

\begin{proof}
Write $F=F(0)+\widetilde F$, where $\widetilde F(0)=0$.
Suppose, for contradiction, that $u\in L^\infty ([0,T];H^s(\R))$ is a solution to \eqref{NLS} with the initial datum $\phi$.
Set
\[
\wt u(t,x) := u(t,x) - t F(0).
\]
Then,
we have
\[
\dt \wt u = -i \dx^2 u + \widetilde F(u,\dx u,\cj u,\cj{\dx u}),
\quad
\wt u(0) = \phi.
\]
Lemma \ref{lem:comp} gives
\[
\widetilde F(u,\dx u,\cj u,\cj{\dx u})\in L^\infty ([0,T];H^{s-1}(\R)).
\]
By the weak formulation,
we have
\[
\wt u(t)
= \phi + \int_0^t \big( -i \dx^2 u + \widetilde F(u,\dx u,\cj u,\cj{\dx u}) \big) (t') dt'
\]
in $\mathcal S'(\mathbb R)$ for almost every $t \in [0,T]$.
Since the right-hand side belongs to $H^{s-2}(\R)$,
this forces
\[
t F(0)
= u(t) - \wt u(t) \in H^{s-2}(\R)
\]
for almost every $t \in [0,T]$.
Since a nonzero constant is not in $H^{s-2}(\R)$,
we have $F(0)=0$.
This contradicts the assumption $F(0) \neq 0$.
\end{proof}

\subsection{Failure of the Mizohata condition}

In the following analysis of the Mizohata obstruction, we assume that
\[
F(0)=0.
\]
By Remark \ref{rem:sold},
we regard
a solution $u$ to \eqref{NLS}
as an $H^s(\R)$-valued weakly continuous function.

\begin{definition}
\label{dfn:As}
For $s>\frac 32$ and $T,M>0$,
define $\Al_s(T,M)$ to be the set of all
$\phi\in H^s(\R)$ admitting a solution $u\in C_w([0,T];H^s(\R))$ to \eqref{NLS}
such that
\[
\| u \|_{L_T^\infty H^s}\le M.
\]
Moreover, we set
\begin{equation}
\Al_s
:=
\bigcup_{m,n \in \N}\Al_s(m^{-1},n).
\label{eq:Nbb3}
\end{equation}
\end{definition}

\begin{remark}
\rm
\begin{enumerate}
\item
No uniqueness of solutions or continuity of a solution map is assumed in the definition of $\Al_s(T,M)$.

\item
$\Al_s$ denotes the set of initial data for which \eqref{NLS} admits a solution on some positive time interval.
\end{enumerate}
\end{remark}

\begin{lemma}
\label{lem:ATM-closed}
Assume that $F(0)=0$ and let $s>\frac 32$.
For every $T,M>0$, the set $\Al_s(T,M)$ is closed in $H^s(\R)$.
\end{lemma}

\begin{proof}
Let $\phi_n\in\Al_s(T,M)$ converge to $\phi$ in $H^s(\R)$, and
choose the corresponding solutions $u_n$ to \eqref{NLS} with the
initial datum $\phi_n$ such that
\[
\sup_{n\in\N}\|u_n\|_{L_T^\infty H^s}\le M.
\]
Since $F(0)=0$,
Remark \ref{rem:sold} and Lemma \ref{lem:comp} give
\begin{align*}
\sup_{n\in\N}
\|\dt u_n\|_{L_T^\infty H^{s-2}}
&\le
\sup_{n\in\N}
\big(
\|u_n\|_{L_T^\infty H^s}
+
\| F(u_n, \dx u_n, \cj{u_n}, \cj{\dx u_n})\|_{L_T^\infty H^{s-1}}
\big)
\\
&\le C(M).
\end{align*}

Since $s>\frac 32$, Lemma \ref{lem:cpt} yields,
after taking a subsequence,
a function $u \in C_w([0,T];H^s(\R))$ such that
\[
\| u \|_{L_T^\infty H^s} \le M,
\quad
u(0) = \phi,
\]
and
\[
u_n\to u,
\quad
\dx u_n\to\dx u
\]
locally uniformly on $[0,T]\times\R$.
Therefore,
\[
F(u_n,\dx u_n,\cj{u_n},\cj{\dx u_n})
\to
F(u,\dx u,\cj u,\cj{\dx u})
\]
locally uniformly on $[0,T]\times\R$.
Thus,
$u$ is a solution to \eqref{NLS} with the initial datum $\phi$
in the sense of Definition \ref{def:sol},
and hence
$\phi\in\Al_s(T,M)$.
Consequently,
$\Al_s(T,M)$ is closed in $H^s(\R)$.
\end{proof}

The following theorem is the key step in proving non-existence.

\begin{theorem}
\label{thm:necessity}
Assume that $F(0)=0$ and that at least one of the following conditions fails:
\[
\Im F_\be(0)=0,
\quad
F_{\cj\al\be}(0)=\cj{F_{\al\be}(0)}.
\]
Then, for $s>\frac 52$ and $T,M>0$,
the interior of $\Al_s(T,M)$ in $H^s(\R)$ is empty.
\end{theorem}

In the proof of Theorem \ref{thm:necessity},
we use the following lemma.

\begin{lemma}
\label{lem:mod}

Let $s \in \R$ and $N\ge1$.
For either choice of sign and every $f \in H^s(\R)$,
we have
\begin{align*}
\| e^{\pm iNx} f \|_{H^s}
\les N^{|s|} \| f \|_{H^s}.
\end{align*}
If $s \ge 0$ and $f \in H^{s+1}(\R)$, then the following estimates also hold:
\begin{align*}
N^s
\|e^{\pm iNx} f \|_{H^{-s}}
&\les
\| f \|_{L^2} + N^{-1}\| f \|_{H^{s+1}},
\\
N^s
\|\dx( e^{\pm iNx} f)\|_{H^{-s}}
&\les
N \| f \|_{L^2} + \| f \|_{H^{s+1}}.
\end{align*}
\end{lemma}

\begin{proof}
Peetre's inequality gives
\[
\jb{\xi+N}^{s}
\les
\jb{\xi}^{s}\jb{N}^{|s|}
\]
for $\xi \in \R$.
With Plancherel's identity,
this gives the first bound.

For the last two bounds, split the frequency region into
$|\xi|\le \frac N2$ and $|\xi|> \frac N2$.
On the first region,
we have
\[
\jb{\xi \pm N}^{-s}
\sim N^{-s}.
\]
On the second,
we have
\[
\jb{\xi \pm N}^{-s}
\les N^{-s-1} \jb{\xi}^{s+1}.
\]
These pointwise bounds and Plancherel's identity
give the assertions.
\end{proof}

\begin{proof}[Proof of Theorem \ref{thm:necessity}]
To show that $\Al_s(T,M)$ has empty interior,
suppose that there exist $\phi \in H^s(\R)$ and $\rho>0$ such that
\begin{equation}
B_{H^s}(\phi,\rho)
\subset
\Al_s(T,M).
\label{eq:AA0}
\end{equation}

\smallskip
\noi\emph{Step 1: choose a smooth function inside the ball.}
Choose $\psi_0\in C_c^\infty (\R)$ with
\[
\|\psi_0-\phi\|_{H^s}< \frac \rho4.
\]
If $\Im F_\be(0)\ne0$, put $\psi := \psi_0$.
Otherwise,
the assumption yields $F_{\cj\al\be}(0) \neq \cj{F_{\al\be}(0)}$.
Set
\[
\al_\ast 
:=
\frac{i}{F_{\al\be}(0) - \cj{F_{\cj\al\be}(0)}}.
\]
Then,
we have
\begin{equation}
\Im(F_{\al\be}(0)\al_\ast  +F_{\cj\al\be}(0) \cj{\al_\ast })
=
\Im \big( (F_{\al\be}(0) - \cj{F_{\cj\al\be}(0)}) \al_\ast  \big)
=1.
\label{kkk1}
\end{equation}
Put
\[
\psi :=\psi_0+\dl\al_\ast \jb{x}^{- \frac 34},
\]
where $\dl>0$ is so small that
\begin{equation}
\|\psi-\phi\|_{H^s}< \frac \rho 2.
\label{eq:AA0a}
\end{equation}
In either case,
we have
$\psi \in H^\infty(\R)$ and $\psi\in B_{H^s}(\phi,\rho)
$.
Set
\begin{equation}
a_0 := F_\be(\psi,\psi',\cj\psi,\cj{\psi'}),
\quad
b_0 := F_{\cj\be}(\psi,\psi',\cj\psi,\cj{\psi'}).
\label{ea:AA-2a}
\end{equation}
Each coefficient is the sum of a constant and a function in $H^\infty (\R)$.
Moreover,
there exist $c>0$, $R \ge 1$, and $\nu\in\{1,-1\}$ such that
\begin{equation}
\nu\Im a_0(x)\ge c x^{- \frac 34}
\label{eq:tail-lower}
\end{equation}
for $x \ge R$.
Indeed,
in the case $\Im F_\be (0) \neq 0$,
$\Im a_0$ is a nonzero constant
outside the support of $\psi_0$.
If $\Im F_\be (0)=0$,
we set
$h(x) := \jb{x}^{-\frac 34}$.
Then,
\[
a_0 = F_\be( \dl\al_\ast  h, \dl\al_\ast h', \dl \cj{\al_\ast } h, \dl \cj{\al_\ast} h')
\]
outside the support of $\psi_0$.
It follows from the Taylor expansion at the origin and \eqref{kkk1} that
\[
\Im a_0(x)
=\dl h(x) + O \big( |h'(x)|+h(x)^2 \big)
=\dl x^{- \frac 34}(1+o(1))
\]
for $x \gg 1$.

\smallskip
\noi\emph{Step 2: construct a wave packet.}
Put
\begin{equation}
d := \frac 1{100(1+\|a_0\|_{L^\infty})}.
\label{AA-rd}
\end{equation}
Fix a nonzero $\chi\in C_c^\infty((-d,d))$.
Let $L,N>0$ be large parameters to be chosen later.
Set
\begin{equation}
\begin{split}
(x_0,y_0)
&:=
\begin{cases}(L,2L), &\text{if } \nu=1,\\
(2L,L),& \text{if } \nu=-1,
\end{cases}
\\
T_N
&:=\frac{L}{2N}.
\end{split}
\label{AA-path}
\end{equation}
We require
\[
L \ge R+d,
\quad
T_N \le \min (T,1).
\]
Note that
\begin{equation}
x_0+2\nu N T_N = y_0,
\label{AA-pa1}
\end{equation}
where $\nu$ is as in \eqref{eq:tail-lower}.

Define
\begin{align}
\G(x)&:=-\frac i2\int_{x_0}^x a_0(y) dy,
\label{AA-test0}
\\
h_N(t,x)&:=N^{s-2} e^{-\cj{\G(x)}}
\chi(x-x_0-2\nu Nt)
e^{-i\nu N(x-x_0)+iN^2t}
\label{AA-test}
\end{align}
for $t \ge 0$.
Note that
\[
\supp h_N(t, \cdot) \subset [L-d,2L+d]
\]
for $t \in [0,T_N]$.

Write
\begin{equation}
A_N (t,x)
:= e^{-\cj{\G (x)}}\chi(x-x_0-2\nu Nt).
\label{AA-testa}
\end{equation}
Let
\[
c_\ast  := 2c (2^{\frac 14}-1)>0,
\]
where $c$ is as in \eqref{eq:tail-lower}.
From \eqref{AA-test0}, \eqref{AA-path}, and \eqref{eq:tail-lower},
we have
\[
\Re \G (y_0)
= \frac 12 \int_{x_0}^{y_0} \Im a_0 (y) dy
\ge c_\ast L^{\frac 14}.
\]
Moreover, \eqref{AA-rd} gives
\begin{align*}
\min_{x \in [L-d,2L+d]} \Re \G(x)
&\ge
\min_{x \in [L-d,L] \cup [2L,2L+d]} \Re \G(x)
\\&
\ge
- \frac d2 \|\Im a_0\|_{L^\infty}
\ge
- \frac 1{200}
.
\end{align*}
From \eqref{AA-testa} and \eqref{AA-pa1},
we obtain
\begin{equation}
\begin{split}
&\sup_{0\le t\le T_N}\|A_N(t)\|_{H^s}\les 1,
\quad
\|A_N(T_N)\|_{H^s}\les
\exp ( - c_\ast   L^{\frac 14} ) .
\end{split}
\label{AA07b}
\end{equation}
Here and hereafter, the implicit constants are independent of $L$, $N$, and $t$,
but may depend on $T$, $M$, $\psi$, and $\chi$.

From \eqref{AA-test} and \eqref{AA-testa},
we have
\begin{equation}
h_N(t,x) = N^{s-2} A_N(t,x) e^{-i\nu N(x-x_0)+iN^2t}.
\label{AA07c}
\end{equation}
Lemma \ref{lem:mod} with \eqref{AA07b} and \eqref{AA07c} gives
\begin{align}
\begin{split}
&\sup_{0 \le t \le T_N}
(\|h_N(t)\|_{H^{-(s-2)}} +N\|h_N(t)\|_{H^{-(s-1)}} )
\les 1,
\\	
&\|h_N(T_N)\|_{H^{-(s-2)}}
\les \exp ( - c_\ast  L^{\frac 14} ).
\end{split}
\label{AA-testbb}
\end{align}
The same bounds hold if $\chi$ is replaced by any of its derivatives in the definition of $h_N$.
Direct differentiation, using $\G'=- \frac i2 a_0$, gives
\begin{equation}
 (\dt+(-i\dx^2+a_0\dx)^\ast )h_N=E_N,
\label{AA-08a}
\end{equation}
where, for either sign of $\nu$,
\begin{align*}
E_N(t,x)
&:=
iN^{s-2} e^{-\cj{\G(x)}}
\chi''(x-x_0-2\nu Nt)e^{-i\nu N(x-x_0)+iN^2t}\\
&\quad
+ \Big( \frac i4\cj{a_0 (x)}^{2}-\frac12 \cj{a_0'(x)} \Big)h_N (t,x).
\end{align*}
Here, $(-i\dx^2+a_0\dx)^\ast = i \dx^2 - \cj{a_0'} - \cj{a_0} \dx$ is the formal adjoint of $-i\dx^2+a_0\dx$ in $L^2(\R)$.
From \eqref{AA-testbb} and Lemma \ref{lem:mod},
we have
\begin{equation}
\sup_{0\le t\le T_N} \|E_N(t)\|_{H^{-(s-2)}}
\les 1.
\label{AA08ad}
\end{equation}

\smallskip
\noi
\emph{Step 3: choose an initial datum.}
Let $\eta>0$ be a small constant to be specified later.
Set
\begin{equation}
\kk_N(x)
:=\eta N^{-s}\chi(x-x_0)e^{-i\nu N(x-x_0)},
\label{AA-09aa}
\end{equation}
where $\nu$ and $x_0$ are as in \eqref{eq:tail-lower} and \eqref{AA-path}, respectively.
Then,
Lemma \ref{lem:mod} yields that
\begin{equation}
\sup _{N \ge 1} \|\kk_N\|_{H^s} \les \eta.
\label{eq:AA09}
\end{equation}
Moreover, a direct calculation yields that
\begin{equation}
\|\kk_N\|_{W^{1,\infty}}
\les
\eta N^{1-s}.
\label{eq:AA09b}
\end{equation}

It follows from \eqref{AA-test} and \eqref{AA-09aa} that
\[
(\kk_N'',h_N(0))_{L^2}
=-\eta\int_{\R} |\chi(x-x_0)|^2e^{-\G(x)} dx
+ O( N^{-1}).
\]
From \eqref{AA-rd} and \eqref{AA-test0},
we have
\[
\sup_{|x-x_0| \le d} |\G(x)|
\le \frac d2 \| a_0 \|_{L^\infty}
\le \frac 1{100}.
\]
Since $\supp \chi \subset (-d,d)$,
there exist $c_0>0$ and $N_0\ge1$, independent of $L$,
such that
\begin{equation}
|(\kk_N'',h_N(0))_{L^2}
|\ge c_0\eta
\label{ea:AA-10}
\end{equation}
for $N\ge N_0$.

For any $\al, \be \in \C$,
we have
\begin{align*}
\max ( |\al+\be|, |\al-\be|)^2
\ge
\frac 12 ( |\al+\be|^2+|\al-\be|^2)
= |\al|^2+ |\be|^2
\ge |\be|^2.
\end{align*}
Hence,
we can choose $\s_N\in\{1,-1\}$ with
\begin{equation}
|(\psi'',h_N(0))_{L^2}
+\s_N (\kk_N'',h_N(0))_{L^2}|
\ge
|(\kk_N'',h_N(0))_{L^2}|.
\label{ea:AA-11}
\end{equation}
Define the single-packet datum
\begin{equation}
\phi_N:
=\psi+\s_N\kk_N.
\label{AA-initial}
\end{equation}
It follows from \eqref{eq:AA0}, \eqref{eq:AA0a}, and \eqref{eq:AA09} that
\begin{equation}
\phi_N
\in B_{H^s}(\phi,\rho)
\subset\Al_s(T,M)
\label{eq:AA11i}
\end{equation}
provided that $\eta$ is sufficiently small.
From \eqref{ea:AA-10} and \eqref{ea:AA-11},
we have
\begin{equation}
|(\phi_N'', h_N(0))_{L^2}|\ge c_0\eta
\label{ea:AA-12}
\end{equation}
for $N \ge N_0$.

\smallskip
\noi\emph{Step 4: estimate a solution.}
By \eqref{eq:AA11i},
there exists a solution
\[
u_N \in C_w([0,T];H^s(\R))
\]
to \eqref{NLS}
with the initial datum $\phi_N$ and
\[
\| u_N \|_{L_T^\infty H^s} \le M.
\]
It follows from \eqref{eq:Duh}, $F(0)=0$, and Lemma \ref{lem:comp} that
\begin{equation}
\begin{split}
\|u_N(t)-\phi_N\|_{H^{s-1}}
&\les
\| e^{-it \dx^2} \phi_N - \phi_N \|_{H^{s-1}}
\\
&\quad
+
\int_0^t \| F(u_N,\dx u_N,\cj{u_N},\cj{\dx u_N}) (t') \|_{H^{s-1}} dt'
\\
&\les
t^{\frac 12}
\end{split}
\label{eq:AA-13}
\end{equation}
for $0<t< \min (T,1)$.

Set
\begin{equation}
v_N:=\dx u_N,
\quad w_N:=\dx^2u_N,
\label{eq:AA-2aa}
\end{equation}
and
\begin{equation}
a_N :=F_\be(u_N,v_N,\cj{u_N},\cj{v_N}),
\quad
b_N :=F_{\cj\be}(u_N,v_N,\cj{u_N},\cj{v_N}).
\label{ea:AA-2}
\end{equation}
As in \eqref{weps},
we write
\begin{equation}
\dt w_N +i \dx^2 w_N =a_N \dx w_N +b_N\dx\cj{w_N}+R_N,
\label{AA-w}
\end{equation}
where $R_N$ is a remainder term satisfying
\begin{equation}
\|w_N\|_{L_T^\infty H^{s-2}}+\|R_N\|_{L_T^\infty H^{s-2}}
\les 1.
\label{eq:AAwR}
\end{equation}

It follows from \eqref{AA-initial}, \eqref{eq:AA09b}, and \eqref{eq:AA-13} that
\begin{align*}
&\| u_N(t) - \psi \|_{L^\infty}
+ \| \dx u_N(t) - \dx \psi \|_{L^\infty}
\\
&\les
\| u_N(t) - \phi_N \|_{H^{s-1}}
+
\| \kk_N \|_{W^{1,\infty}}
\\
&\les
N^{1-s}+t^{\frac 12}.
\end{align*}
By \eqref{ea:AA-2a}, \eqref{ea:AA-2}, and the mean value theorem applied to $F_\be$ and $F_{\cj\be}$,
we have
\begin{equation}
\begin{split}
&
\|a_N(t)-a_0\|_{L^\infty}+\|b_N(t)-b_0\|_{L^\infty}
\\
&\les
\| u_N(t) - \psi \|_{L^\infty}
+ \| \dx u_N(t) - \dx \psi \|_{L^\infty}
\\
&\les
N^{1-s}+t^{\frac 12}.
\end{split}
\label{AA-delta}
\end{equation}

For $T_N\le T$,
from \eqref{AA-08a} and \eqref{AA-w},
we have
\begin{align}
\begin{split}
\frac d{dt} (w_N,h_N)_{L^2}
&=
(-i \dx^2 w_N + a_N \dx w_N + b_N \dx \cj{w_N} +R_N, h_N)_{L^2}
\\
&\quad
+ (w_N,\dt h_N)_{L^2}
\\
&=
((a_N-a_0)\dx w_N,h_N)_{L^2}
+
((b_N-b_0)\dx \cj{w_N},h_N)_{L^2}
\\
&\quad
+ (R_N,h_N)_{L^2}
+ (w_N,E_N)_{L^2}
+ (b_0 \dx \cj{w_N},h_N)_{L^2}
\\
&=:
\I_1
+
\I_2
+
\I_3
+
\I_4
+
\I_5
.
\end{split}
\label{AA17}
\end{align}
Note that
\[
w_N \in C_w([0,T]; H^{s-2}(\R)),
\quad
\dt w_N \in L^\infty ([0,T]; H^{s-4}(\R)).
\]
Since \eqref{AA-test} gives $h_N \in C^1([0,T];H^\infty(\R))$,
we interpret $(w_N,h_N)_{L^2}$ as the duality pairing between $H^{s-4}(\R)$ and $H^{-(s-4)}(\R)$.
Then,
$(w_N,h_N)_{L^2}$ is absolutely continuous on $[0,T_N]$,
and \eqref{AA17} holds for almost every $t\in[0,T_N]$.

We have
\begin{equation}
|I_j(t)| 
\les
N t^{\frac 12} +1
\label{AA18}
\end{equation}
for $j=1,2,3,4$.
Indeed,
Lemmas \ref{lem:mod}, \ref{lem:product}, and \ref{lem:comp} with
\eqref{AA07c}, \eqref{AA07b}, \eqref{AA-delta}, and $s>\frac 52$
yield
\begin{align*}
&\|\dx((\cj{a_N(t)-a_0}) h_N(t))\|_{H^{-(s-2)}}
\\
&\les
N\| (\cj{a_N(t)-a_0}) A_N(t) \|_{L^2}
+\| (\cj{a_N(t)-a_0}) A_N(t) \|_{H^{s-1}}
\\
&\les
N( N^{1-s}+t^{\frac 12})+1
\\
&\les
N t^{\frac 12}+1.
\end{align*}
From \eqref{AA17}, we have
\begin{align*}
|I_1(t)|
&= |(w_N(t),\dx ((\cj{a_N(t)-a_0}) h_N(t)))_{L^2}|
\\
&\le
\| w_N (t) \|_{H^{s-2}} \| \dx ( (\cj{a_N(t)-a_0}) h_N(t)) \|_{H^{-(s-2)}}
\\
&\les
N t^{\frac 12} +1.
\end{align*}
The same calculation yields that
\[
|I_2(t)|
\les
N t^{\frac 12} +1.
\]
It follows from \eqref{AA17}, \eqref{AA-testbb}, and \eqref{eq:AAwR} that
\[
|I_3(t)|
\le
\| R_N(t) \|_{H^{s-2}} \| h_N(t) \|_{H^{-(s-2)}}
\les
1.
\]
Moreover,
\eqref{AA17}, \eqref{eq:AAwR}, and \eqref{AA08ad} give
\[
|I_4(t)|
\le
\| w_N(t) \|_{H^{s-2}} \| E_N(t) \|_{H^{-(s-2)}}
\les 1.
\]
Combining the estimates above,
we obtain \eqref{AA18}.

We now estimate $I_5$ in \eqref{AA17}.
Since $u_N$ is a solution to \eqref{NLS},
\eqref{eq:AA-2aa} and \eqref{ea:AA-2} give
\[
\dt \cj{v_N} = i\dx\cj{w_N}+H_N,
\]
where
\[
H_N
:=\cj{a_N} \cj{w_N}+\cj{b_N} w_N
+\cj{F_\al(u_N,v_N,\cj{u_N},\cj{v_N})} \cj{v_N}
+\cj{F_{\cj\al} (u_N,v_N,\cj{u_N},\cj{v_N})} v_N.
\]
Corollary \ref{cor:comppro}  gives
\begin{equation}
\sup_{0 \le t \le T_N}
\|H_N (t)\|_{H^{s-2}}
\les 1.
\label{AA:20a}
\end{equation}
Using \eqref{AA-08a},
we obtain
\begin{equation}
\begin{split}
\frac d{dt} (b_0 \cj{v_N},h_N)_{L^2}
&=
2i (b_0\dx\cj{w_N},h_N)_{L^2}
+ (Q_N,h_N)_{L^2}
+ (b_0 \cj{v_N},E_N)_{L^2},
\end{split}
\label{eq:AA21b}
\end{equation}
where
\[
Q_N
:=
b_0 H_N +2i b_0' \cj{w_N}
+i b_0'' \cj{v_N}
-a_0 ( b_0\cj{w_N} + b_0' \cj{v_N} ).
\]
It follows from Lemma \ref{lem:product}, \eqref{ea:AA-2a}, and \eqref{AA:20a} that
\[
\sup_{0 \le t \le T_N}
\|Q_N(t)\|_{H^{s-2}}
\les 1.
\]
By \eqref{ea:AA-2a} and \eqref{AA-testbb}, we have
\[
\sup_{0 \le t \le T_N}
|(b_0\cj{v_N(t)},h_N(t))_{L^2}|
\le
\sup_{0 \le t \le T_N}
\|b_0\cj{v_N(t)}\|_{H^{s-1}} \|h_N(t)\|_{H^{-(s-1)}}
\les
N^{-1}.
\]
Moreover,
it follows from \eqref{ea:AA-2a} and \eqref{AA08ad} that
\[
\sup_{0 \le t \le T_N}
|(b_0 \cj{v_N(t)},E_N(t))_{L^2}|
\le
\sup_{0 \le t \le T_N}
\|b_0\cj{v_N(t)}\|_{H^{s-2}} \|E_N(t)\|_{H^{-(s-2)}}
\les
1.
\]
Integrating \eqref{eq:AA21b} over $[0,T_N]$,
we obtain
\begin{equation}
\bigg| \int_0^{T_N} (b_0 \dx \cj{w_N},h_N)_{L^2} dt \bigg|
\les
N^{-1}+T_N.
\label{AA-21}
\end{equation}

Combining \eqref{AA17}, \eqref{AA18}, \eqref{AA-21},
and \eqref{AA-path},
we find
\begin{align}
\begin{split}
&| (w_N(T_N), h_N(T_N))_{L^2} - (w_N(0),h_N(0))_{L^2}|
\\
&=
\bigg|
\int_0^{T_N} \frac d{dt} (w_N,h_N)_{L^2} dt \bigg|
\\
&\les
T_N (N T_N^{\frac 12}+1) + N^{-1} + T_N
\\
&\les
L^{\frac 32}N^{-\frac 12}
.
\end{split}
\label{AA:22}
\end{align}
On the other hand,
from \eqref{AA-testbb},
we have
\begin{equation}
\begin{aligned}
|(w_N(T_N),h_N(T_N))_{L^2}|
&\le
\| w_N(T_N) \|_{H^{s-2}} \| h_N(T_N) \|_{H^{-(s-2)}}
\\
&\les
\exp ( - c_\ast  L^{\frac 14} )
.
\end{aligned}
\label{AA:23}
\end{equation}

\smallskip
\noi\emph{Step 5: choose parameters and conclude.}
Choose $L>0$ sufficiently large and set
\[
N
:=
L^4.
\]
With \eqref{AA-path}, we have
\[
T_N = \frac 1{2L^3} \le \min (T,1)
\]
for $L \gg 1$.

Thus, \eqref{AA:22} and \eqref{AA:23} yield that
\begin{align*}
|(w_N(0),h_N(0))_{L^2}|
&\le
|(w_N(T_N),h_N(T_N))_{L^2}|
\\
&\quad
+
|(w_N(T_N),h_N(T_N))_{L^2}
- (w_N(0),h_N(0))_{L^2} |
\\
&\les
\exp ( - c_\ast  L^{\frac 14} )
+
L^{-\frac 12}
.
\end{align*}
This contradicts \eqref{ea:AA-12}
provided that $L \gg 1$.
Thus,
$\Al_s(T,M)$ has empty interior.
\end{proof}

\begin{corollary}
\label{cor:ABS}
Assume that $F(0)=0$ and that at least one of
\[
\Im F_\be(0)=0,
\quad
F_{\cj\al\be}(0)=\cj{F_{\al\be}(0)}
\]
fails.
Then, for $s>\frac 52$,
the interior of $\Al_s$ in $H^s(\R)$ is empty.
\end{corollary}

\begin{proof}
With \eqref{eq:Nbb3},
we have
\[
H^s(\R) \setminus \Al_s
=
\bigcap_{m,n \in \N}
( H^s(\R)\setminus\Al_s(m^{-1},n) ).
\]
By Lemma \ref{lem:ATM-closed} and Theorem \ref{thm:necessity},
each $\Al_s(m^{-1},n)$ is closed and has empty interior in $H^s(\R)$.
The Baire category theorem therefore implies that
$H^s(\R) \setminus \Al_s$ is dense in $H^s(\R)$.
Hence,
$\Al_s$ has empty interior.
\end{proof}

The proof also shows that $H^s(\R) \setminus \Al_s$ is a dense $G_\delta$ subset of $H^s(\R)$.
Moreover,
under the same assumptions as in Corollary \ref{cor:ABS},
for $\phi \in H^s(\R) \setminus \Al_s$ and $T>0$,
there is no distributional solution $u \in C_w([0,T]; H^s(\R))$ to \eqref{NLS}.

\begin{proof}[Proof of Theorem \ref{thm:main}]
The implication (iii)$\implies$(i) follows from Theorem \ref{thm:suff},
and (i)$\implies$(ii) is immediate.

For (ii)$\implies$(iii),
we prove the contrapositive.
If $F(0)\ne0$,
Proposition \ref{prop:F0} shows that there is no $H^s(\R)$-solution to \eqref{NLS} for any initial datum in $H^s(\R)$.
When $F(0)=0$,
Remark \ref{rem:sold} shows that (ii) is equivalent to $\Al_s$ having nonempty interior in $H^s(\R)$.
If $F(0)=0$ but (iii) fails,
then at least one of
\[
\Im F_\be(0)=0,
\quad
F_{\cj\al\be}(0)=\cj{F_{\al\be}(0)}
\]
fails.
Corollary \ref{cor:ABS} gives
that the interior of $\Al_s$ is empty,
which proves the contrapositive.
\end{proof}

\appendix
\section{Strong continuity in time}
\label{SEC:ASCT}

In this appendix, we prove the following proposition.

\begin{proposition}
Let $s>\frac 32$.
Assume that $F(0)=0$ and that \eqref{NLS} is locally well-posed
in $H^s(\R)$ in the sense of Definition \ref{def:wp}.
Let $T>0$ and $U$ be as in Definition \ref{def:wp}.
Then, for every $\phi\in U$, the solution $u^{[\phi]}$,
with the representative specified in Remark \ref{rem:sold},
belongs to $C([0,T];H^s(\R))$.
\end{proposition}

\begin{proof}
Fix $\phi\in U$ and write $u := u^{[\phi]}$.

\smallskip
\noi\emph{Step 1: spatial translations.}
For $h\in\R$, define
\[
\tau_h f(x):=f(x+h).
\]
Since $U$ is open and
\[
\lim_{h \to 0}\|\tau_h\phi-\phi\|_{H^s}=0,
\]
we have $\tau_h \phi\in U$ for all sufficiently small $|h|$.
For such $h$, set
\[
u_h:=u^{[\tau_h\phi]}.
\]
By the spatial translation invariance of \eqref{NLS},
$\tau_h u$ satisfies
\[
\begin{cases}
\dt(\tau_h u)+i\dx^2(\tau_h u)
=
F(\tau_h u,\dx(\tau_h u),
\cj{\tau_h u},\cj{\dx(\tau_h u)}),
\\
(\tau_h u)(0,x)=(\tau_h \phi)(x).
\end{cases}
\]
Thus, $\tau_h u$ and $u_h$ solve the same Cauchy problem
on $[0,T]$ in the sense of Definition \ref{def:sol}.
Uniqueness gives $u_h=\tau_h u$.
The continuity of the solution map therefore implies
\[
\lim_{h\to0}
\|\tau_h u-u\|_{L_T^\infty H^s}=0.
\]

\smallskip
\noi\emph{Step 2: spatial regularization.}
Choose $\chi\in C_c^\infty((-1,1))$ such that
\[
\chi\ge0,
\quad
\int_\R\chi(x) dx=1.
\]
For $\eta>0$, define
\[
u_\eta(t,x)
:=
\int_\R\frac1\eta
\chi \Big( \frac{x-y}{\eta} \Big) u(t,y) dy
=
\int_\R\chi(y)(\tau_{-\eta y}u)(t,x) dy.
\]
Then,
\[
\begin{aligned}
\|u_\eta-u\|_{L_T^\infty H^s}
&\le
\int_\R\chi(y)
\|\tau_{-\eta y}u-u\|_{L_T^\infty H^s} dy
\\
&\le
\sup_{|h|\le\eta}
\|\tau_h u-u\|_{L_T^\infty H^s}.
\end{aligned}
\]
By Step 1, it follows that
\[
\lim_{\eta \downarrow 0}
\|u_\eta-u\|_{L_T^\infty H^s}=0.
\]

By Remark \ref{rem:sold},
\[
u\in C([0,T];H^{s-1}(\R))
\cap C_w([0,T];H^s(\R)).
\]
Spatial regularization therefore gives
\[
u_\eta\in C([0,T];H^s(\R))
\]
for every $\eta>0$.
Moreover,
$u_\eta-u\in C_w([0,T];H^s(\R))$.
With \eqref{NNN1},
we have
\[
\sup_{t\in[0,T]}\|u_\eta(t)-u(t)\|_{H^s}
=
\|u_\eta-u\|_{L_T^\infty H^s}
\]
for every $\eta>0$.
The right-hand side tends to zero as $\eta\downarrow0$.
Thus, $u$ is the uniform limit on $[0,T]$ of
$H^s(\R)$-valued continuous functions.
Consequently,
$u\in C([0,T];H^s(\R))$.
\end{proof}

\mbox{}

\noi
{\bf 
Acknowledgements.}
This work was
supported by JSPS KAKENHI Grant Number
JP26K00613 and JST SPRING Grant Number JPMJSP2132.

\end{document}